\documentclass[12pt, twoside, leqno]{article}

\usepackage{graphicx} 

\usepackage{amsmath,amsthm}
\usepackage{amssymb}

\usepackage{enumitem}

\usepackage{graphicx}

\usepackage{tikz}

\newtheorem{theorem}{Theorem}[section]
\newtheorem{corollary}[theorem]{Corollary}
\newtheorem{lemma}[theorem]{Lemma}
\newtheorem{proposition}[theorem]{Proposition}

\theoremstyle{definition}
\newtheorem{definition}[theorem]{Definition}
\newtheorem{remark}[theorem]{Remark}
\newtheorem{example}[theorem]{Example}

\newtheorem*{xrem}{Remark}

\numberwithin{equation}{section}

\begin{document}


\baselineskip=17pt


\title{On a Constraint on Invariant Measures of Certain Cellular Automata (Bi-Permutative and More)}
\author{Matan Tal\\
The Hebrew University of Jerusalem}
\date{}
\maketitle

\vspace{1cm}

\begin{abstract}
    In  \cite{key-1}, a constraint on invariant measures of bi-permutative cellular automata has been observed. Its content was more or less as follows: fixed values at the positive indices determine almost-surely a uniform conditional probability on the subset of values of positive conditional probability at the zero index. When the alphabet is a finite group and the automaton is multiplication of two neighbors, that set is in fact a coset of some subgroup. In the present paper, we strengthen the formulations in \cite{key-1} and investigate further the implications of this constraint. In the finite group case mentioned above, relations between some attributes of the group structure and the invariant measures are examined. We also inspect a factor that this constraint induces (we called it the structural factor), and analyze the special case in which it has zero measure-theoretical entropy, thus observing an interplay between existence of zero entropy invariant measures on that factor and existence of positive entropy measures corresponding to them on the original system. Then, we leave the setting of bi-permutative cellular automata and generalize our results to a wider class which we named RLP subshifts. The peculiar situation is that although this class may be much larger than the class of bi-permutative cellular automata, we found essentially only one other example in this class - the symbolic coding of the times 2 times 3 system on the circle (and its generalizations).
\end{abstract}

\tableofcontents

\vspace{2cm}

\section{Introduction}

 Let $\Lambda$ be some finite alphabet (i.e. finite set). A cellular automaton is any continuous map $ \Lambda^{\mathbb{N}_0} \rightarrow \Lambda^{\mathbb{N}_0}$ ($\mathbb{N}_0$ stands for the non-negative integers) that commutes with the shift.\\
 
 In most of this paper we shall be concerned with bi-permutative cellular automata. \\

\begin{definition}
\label{biper_def}
    Let $\tau : \Lambda^{\mathbb{N}_0} \rightarrow \Lambda^{\mathbb{N}_0}$ be a cellular automaton defined as \\ $\left( \tau\left( x  \right) \right) _i  = r \left( x_i,x_{i+1} \right)$  for some function $r : \Lambda^2 \rightarrow \Lambda$. If for every fixed $a\in \Lambda$ the single variable function $r\left( \cdot,a \right):\Lambda\rightarrow \Lambda$ is bijective then the automaton is said to be \textit{left permutative}. \textit{Right permutativity} is similarly defined with $r\left(a,\cdot\right):\Lambda\rightarrow \Lambda$. If $\tau$ is both right and left permutative then it is said to be \textit{bi-permutative}. 
\end{definition}

\vspace{0.3cm}

Our object of interest is shift-invariant Borel probability measures on $\Lambda^{\mathbb{N}_0}$ that are also invariant relative to some fixed bi-permutative cellular automaton. Apart from atomic measures, it is well known that the uniform measure is such a measure for any surjective cellular automaton (as it is the unique measure of maximal entropy relative to the shift). What other such measures exist is still quite a mystery (this is a known open problem for cellular automata \cite{boyle,kari-taati,piv-ref3}). Our aim will be to develop further a specific constraint on them that was observed by M. Pivato in \cite{key-1}.\\

Section \ref{pivato_sect}  is dedicated to present the relevant results from  \cite{key-1}, while in Section \ref{basic_theorem_sect} we prove our improvement of the constraint. This is Theorem \ref{biper_th} whose main content is as follows: given a bi-permutative cellular automaton and such a measure, the conditional probability of the value at index zero given the values at the positive indices is almost-surely the uniform probability on some subset of $\Lambda$ (which depends on those values), and moreover, the size of this subset is invariant relative to the action of that cellular automaton. The improvements relative to the original formulation in \cite{key-1} is explained in detail in Section \ref{pivato_sect}, in particular, the ergodicity assumption on the action of the cellular automaton was removed.\\

A class of bi-permutative automata to which we shall give special attention is the class of group multiplication cellular automata of finite groups.\\

\begin{definition}
\label{group_def}
Given a finite group $G$ (not necessarily Abelian), we denote by $\tau_{G}:G^{{\mathbb{N}_0}}\rightarrow G^{{\mathbb{N}_0}}$ \textit{the group multiplication
cellular automaton}, i.e. the cellular automaton determined by the rule $\left(\tau_{G}\left(x\right)\right)_{i}=x_{i}\cdot x_{i+1}$.

\end{definition}

\vspace{0.3cm}

For an Abelian group $G$, the cellular automaton $\tau_G$ is known to have many such invariant measures of positive measure-theoretical entropy relative to each of the two maps (cf. \cite{key-4}) but also displays some rigidity for such measures, for example, when various kinds of total ergodicity or mixing are assumed together with positivity of entropy (for rigidity results in the Abelian case see \cite{key-7,key-5,piv-ref3,sablik,sobottka}). It is yet unknown whether there exist atom-less invariant measures of zero measure-theoretical entropy (relative to each of the two maps).\\

As for non-Abelian groups $G$, it seems that  essentially there are no known examples of atom-less shift-invariant Borel probability measures on $G^{\mathbb{N}_0}$ that are also invariant relative to $\tau_G$ other than the uniform measure on $H^{\mathbb{N}_0}$ - for some $H\leq G$ - and constructions of such measures taking advantage of Abelian subgroups or Abelian quotient groups. Let us also mention that for a group of the form $G_1\times\dots\ \times G_n$, one can take the product of measures of the kinds just mentioned and also atomic ones (the example given in \cite{key-1} is of this kind) - however, ergodicity might be lost on the way.\\

Despite our lack of competence to construct examples for such measures in the case of non-Abelian groups, there is yet no reason to believe that such measures are not as abundant as in the Abelian case.\\

As $\tau_G$ is a bi-permutative cellular automaton, the previously mentioned constraint on an invariant measure applies, but additional information can be obtained. This was observed in \cite{key-1} (its original formulation is presented in Section \ref{pivato_sect}), and our improved version appears in Section \ref{group_sect} as Theorem \ref{group_th}. This theorem says that the subset of the above mentioned conditional probability is in fact a coset of some subgroup of $G$, and that this subgroup is invariant relative to the action of $\tau_G$. The improvement here is that the ergodicity assumption on the action of $\tau_G$ is now omitted. The corollaries then drawn with respect to invariant measures are also related to some attributes of the group $G$.\\

The  constraint on the invariant measures of bi-permutative cellular automata  naturally leads to a definition of a certain factor relative to the shift action that we called the structural factor (denoted $\pi_\mu$). It is defined in Sub-Section \ref{pi_subsect}. The special case in which the structural factor possesses zero entropy is discussed in Sections \ref{group_sect} and \ref{zero_ent_section} (the latter is dedicated exclusively to this case).\\

Up to section \ref{zero_ent_section} we deal solely with bi-permutative cellular automata. In Section \ref{sect_RLP} we generalize our results to a larger class which we named RLP subshifts. Although this class may be much larger than the class of bi-permutative cellular automata, we could identify only one type of example which is an RLP subshift and not a bi-permutative cellular automaton. These are the cellular automata originating in a symbolic coding of the times 2 times 3 multiplication system on the circle $\mathbb{R}/\mathbb{Z}$ (and its generalizations). For this system we anyway have in our disposal Rudolph's Theorem -  which states that the only invariant measure which is ergodic relative to the joint action of the two maps and possesses positive entropy relative to each of them is the uniform measure - and hence our inquiry does not yield an extra insight.\\

The following question is in our eyes perhaps the most unsettling basic open problem regarding invariant measures of cellular automata. We fell short of being able to answer it.\\

\underline{Open Problem:} Are there cellular automata, other than the ones originating from symbolic codings of the times 2 times 3 map on the circle and its generalizations, for which the uniform measure is the only shift-invariant measure that is also invariant relative to the automaton ergodic relative to the joint action of the shift and the automaton that possesses positive entropy relative to each of the two maps? Maybe some requirement on RLP subshifts can guarantee this property?\\

\underline{Acknowledgements:} The author thanks Prof. H. Furstenberg, Prof. T. Meyerovitch and the anonymous referee for their valuable remarks.

\vspace{0.4cm}

\section{Notation and Conventions}

\begin{itemize}
    \item ${\mathbb{N}_0}$ is the non-negative integers. 

    \item $\Lambda$ is some finite alphabet (i.e. a finite set).

    \item $\sigma$ is the left shift map on $\Lambda^{\mathbb{N}_0}$ and, by abuse of notation, also on other alphabets.

    \item For the two-dimensional shift space $\Lambda^{{\mathbb{N}_0}^2}$, we define $\left(\sigma_{\left(1,0\right)}\left(x\right)\right)_{i,j} := x_{i+1,j}$ and $\left(\sigma_{\left(0,1\right)}\left(x\right)\right)_{i,j} := x_{i,j+1}$.

    \item For $a_0,\dots,a_n\in \Lambda$, the cylinder $\{x\in\Lambda^{\mathbb{N}_0}\,:\,x_0=a_0,\dots,x_n=a_n\}$ is denoted by $\left[a_0,\dots,a_n\right]$.

    \item When $\mathcal{A}$ is a countably generated $\sigma$-algebra on a space $X$, we denote the atom of $\mathcal{A}$ to which $x\in X$ belongs to by $\left[ x \right]_\mathcal{A}$.

    \item Let $\nu$ be a probability measure  on a standard Borel space $\left(X,\mathcal{B}\right)$ (i.e. a measurable space which is isomorphic to a Borel subset of a compact metric space) and $\mathcal{A}$ a sub-$\sigma$-algebra of $\mathcal{B}$. Given $f\left(x\right)$ in $L_\nu^1\left(X\right)$, we denote the conditional expectation of $f$ with respect to $\mathcal{A}$ by $\mathbb{E}_\nu \left( f | \mathcal{A} \right)\left(x\right)$. We denote the conditional probability of $\nu$ with respect to $\mathcal{A}$ by $\nu_{x}^{\mathcal{A}}$. For example, considering a measure preserving system $\left(\Lambda^{\mathbb{N}_0},\mathcal{B},\nu,\sigma\right)$ (for some shift-invariant Borel probability measure $\nu$), the conditional measure $\nu_{x}^{\sigma^{-1}\mathcal{B}}$ stands for conditioning on the event composed of all $y$ for which $\sigma\left(y\right) = \sigma \left(x\right)$.

    \item Continuing in the setting of the previous bullet, let $\xi$ be a countable partition of $X$ into measurable sets. For $x\in X$, we denote by $\xi \left(x\right)$ the element of $\xi$ that contains $x$. The conditional entropy of $\xi$ relative to $\mathcal{A}$ is defined as $H_\nu\left(\xi|\mathcal{A}\right) =- \int \log\nu_x^\mathcal{A}\left(\xi\left(x\right)\right) \, d\nu$. For $T:X\rightarrow X$ that preserves $\nu$, the conditional entropy of $\xi$ and $T$ relative to $\mathcal{A}$ is defined as $h_\nu \left(\xi,T|\mathcal{A}\right) = \lim_{N\rightarrow \infty} \frac{1}{N}H_\nu\left(\vee_{i=0}^{N-1} T^{-i}\xi|\mathcal{A}\right)$, and the conditional entropy of $T$ relative to $\mathcal{A}$ as $h_\nu\left(T|\mathcal{A}\right) = \sup_{\{\xi:H\left(\xi\right)<\infty\}} h_\nu \left(\xi,T|\mathcal{A}\right)$.
    
    \item Equalities between functions defined on a probability space should be understood to hold only almost-surely.

    \item We work throughout in the category of measure-preserving systems. In particular, given such a system, a factor is a measure preserving map from it to another such system that is equivariant relative to their maps.

    \item For a measure space $\left(X,\mathcal{B}_X,\nu\right)$, we say that one sub-$\sigma$-algebra of $\mathcal{B}_X$ is contained \textit{contained modulo $\nu$} in another, if for every set of the first there exists a set of the second so that their symmetric difference is a $\nu$-null set. Two sub-$\sigma$-algebras of $\mathcal{B}_X$ are \textit{equal modulo $\nu$} if each of them is contained modulo $\nu$ in the other. We say that a function $f$ from $X$ to a measurable space $\left(Y,\mathcal{B}_Y\right)$ is measure-theoretically measurable with respect to a sub-$\sigma$-algebra $\mathcal{A}\subseteq \mathcal{B}_X$ if $f^{-1}\mathcal{B}_Y$ is contained modulo $\nu$ in $\mathcal{A}$.

\end{itemize}

\section{Our Point of Departure: The Two Theorems of Pivato}

\label{pivato_sect}

Our point of departure is two theorems proved by M. Pivato in \cite{key-1}. Before we formulate them and include their short proofs, a general remark concerning bi-permutative cellular automata is in order.\\

\begin{remark}
\label{biper}
    Given a right permutative cellular automaton $\tau : \Lambda^{\mathbb{N}_0} \rightarrow \Lambda^{\mathbb{N}_0}$, the values $\left( \left( \tau^i \left( x \right) \right) _0 \right)_{i=0}^\infty$ determine  $\left( \left( \tau^i \left( \sigma \left(x\right) \right) \right) _0 \right)_{i=0}^\infty$ also by a right permutative cellular automaton $\hat{\tau}$ often called \textit{the transpose cellular automaton}. The map from $ \Lambda^{\mathbb{N}_0}$ to itself defined by $x\mapsto\left( \left( \tau^i \left( x \right) \right) _0 \right)_{i=0}^\infty$ is thus an isomorphism between the topological dynamical systems $\left(  \Lambda^{\mathbb{N}_0},\sigma,\tau \right)$ and $\left(  \Lambda^{\mathbb{N}_0},\hat{\tau},\sigma \right)$. Moreover, $\tau$ is bi-permutative if and only if $\hat{\tau}$ is.
\end{remark}

\vspace{1cm}

\subsection{Bi-Permutative Cellular Automata}

The first of the two Theorems (which appears in \cite{key-1} as Theorem 4.1) is essentially the following:\\

\begin{theorem}
\label{pivato1}
Let $\tau:\Lambda^{{\mathbb{N}_0}}\rightarrow\Lambda^{{\mathbb{N}_0}}$
be a left permutative cellular automaton, and let $\mu$ be a $\tau$-invariant Borel probability measure on $\Lambda^{{\mathbb{N}_0}}$ which is ergodic relative to $\tau$. Form the two-dimensional subshift $\left\{x\in\Lambda^{{\mathbb{N}_0}\oplus\mathbb{Z}}\,:\, \forall i\in\mathbb{Z} \,\, x_{\left(\cdot\right),i+1} = \tau\left(x_{\left(\cdot\right),i}\right)\right\}$, denote its Borel $\sigma$-algebra by $\mathcal{\tilde{B}}$, and extend $\mu$ to the measure $\tilde{\mu}$ on this subshift. Then there exists  $1\leq k_{\mu} \leq\left|\Lambda\right|$ for which $\tilde{\mu}_{x}^{\sigma_{\left(1,0\right)}^{-1}\mathcal{\tilde{B}}}\left(\left\{ x\right\} \right)=\frac{1}{k_{\mu}}$
$\tilde{\mu}$-almost-surely.

\end{theorem}

\begin{xrem}
The original formulation was weaker but its proof proves the result stated above. There are two main differences:\\
\begin{enumerate}
\item In the original formulation the roles of the one-dimensional
shift and $\tau$ are reversed (in particular, the two dimensional
subshift is then $\{ x\in\Lambda^{\mathbb{Z}\oplus{\mathbb{N}_0}}\,:\,\forall i\in\mathbb{Z\,}~~x_{\left(\cdot\right),i+1} = \tau\left(x_{\left(\cdot\right),i}\right)\} $
and $\tau$ is required to be bi-permutative. The formulation here
is stronger because the system of $\Lambda^{{\mathbb{N}_0}}$ is isomorphic to $\{ x\in\Lambda^{{\mathbb{N}_0}\oplus{\mathbb{N}_0}}\,:\,\forall i\in\mathbb{Z\,}~~x_{\left(\cdot\right),i+1} = \tau\left(x_{\left(\cdot\right),i}\right)\}$ and -  for bi-permutative automata - also each $x_{j,\left(\cdot\right)}$ column determines the column $x_{j+1,\left(\cdot\right)}$ by a bi-permutative cellular automaton (cf. Remark \ref{biper}) and this corresponds to the effect of $\sigma_{\left(1,0\right)}$.
\item  In the original formulation the premise required the measure to be invariant relative to both $\sigma$ and $\tau$ whereas the proof does not utilize this.

\end{enumerate}

\end{xrem}

\begin{proof}
Notice that $\sigma$ and $\tau$ corresponds to $\sigma_{\left(1,0\right)}$ and $\sigma_{\left(0,1\right)}$, respectively, in the two-dimensional subshift.\\

We consider the conditional probabilities with respect to $\sigma_{\left(0,1\right)}^{-1} \mathcal{\tilde{B}}$, that is, the probability conditioned on atoms of that $\sigma$-algebra: the sets of points that share the same values on all indices except maybe the vertical axis.\\

The invariance of $\tilde{\mu}$ and $\sigma_{\left(1,0\right)}^{-1} \mathcal{\tilde{B}}$ relative to $\sigma_{\left(0,1\right)}$ imply that

\[
\tilde{\mu}_{x}^{\sigma_{\left(1,0\right)}^{-1}\mathcal{\tilde{B}}}\left(\left\{ x\right\} \right) =\tilde{\mu}_{\sigma_{\left(0,1\right)}\left(x\right)}^{\sigma_{\left(1,0\right)}^{-1}\mathcal{\tilde{B}}}\left(\left\{ \sigma_{\left(0,1\right)}\left(x\right)\right\} \right),
\]

i.e. $x\mapsto\tilde{\mu}_{x}^{\sigma_{\left(1,0\right)}^{-1}\mathcal{\tilde{B}}}\left(\left\{ x\right\} \right)$ is a $\sigma_{\left(0,1\right)}$-invariant function. Hence by the ergodicity of $\sigma\left(0,1\right)$ relative to $\tilde{\mu}$
there exists a real number $0\leq r\leq1$ so that $\tilde{\mu}_{x}^{\sigma_{\left(1,0\right)}^{-1}\mathcal{\tilde{B}}}\left(\left\{ x\right\} \right)=r$ $\tilde{\mu}$-almost-surely. \\

Knowing $\sigma_{\left(1,0\right)}\left(x\right)$ and $x_{0,0}$ entirely determines $x$ ($x_{0,i}$ is detrmined for positive $i$ by the rule of $\tau$ and for negative
$i$ by left permutativity). Hence the atom contains $\left| \Lambda \right|$ points and there is an integer $1\leq k_\mu \leq\left|\Lambda\right|$
so that $k_\mu \cdot r=1$, i.e. $r=\frac{1}{k_\mu}$. 
\end{proof}

From this theorem, assuming the measure is also $\sigma$-invariant, an important corollary was drawn, stating that the measure-theoretical entropy (relative to both the shift and the automaton) is $\log k_\mu$ for some integer $1 \leq k_\mu \leq \left| \Lambda \right|$.\\

A generaliztion to Theorem \ref{pivato1} with $k_{\mu,x}$ in place of $k_\mu$ is false if the ergodicity assumption is omitted. The reason that the proof fails is that then the ergodic decomposition relative to $\sigma_{\left(0,1\right)}$ need not be measure-theoretically $\sigma_{\left(1,0\right)}^{-1}\tilde{\mathcal{B}}$-measurable.\\

However, if we assume additionally that $\tau$ is bi-permutative and that $\mu$ is also shift-invariant (this is anyway the case that is in our focus of attention), theorem can then be improved. This result will be stated and proved in the next section as Theorem \ref{biper_th}.\\

\subsection{Group Multiplication Cellular Automata}

The second theorem from \cite{key-1} (appears there as Theorem 3.1) concerns multiplication cellular automata of finite groups (not necessarily Abelian).

\begin{theorem}
\label{pivato2}

Let $G$ be a finite group and let $\mu$ be an $\tau_G$-invariant Borel probability measure on $G^{{\mathbb{N}_0}}$ which is ergodic relative to $\tau_G$. Then there exists a  subgroup $ H_{\mu} \leq G$ so that $\mu_{x}^{\sigma^{-1}\mathcal{B}}$ ($\mathcal{B}$ being the Borel $\sigma$-algebra of $G^{\mathbb{N}_0}$) as a measure on $G$ is $\mu$-almost-surely the uniform measure on a right $H_\mu$-coset.

\end{theorem}

\begin{xrem}
The difference between Theorem \ref{pivato2} and the formulation in \cite{key-1} is that in the latter the premise required the measure to be invariant relative to both $\sigma$ and $\tau$ whereas the proof does not utilize this.
\end{xrem}

\begin{proof}

By the ergodicity of $\mu$ relative to $\tau_{G}$, the points in $G^{{\mathbb{N}_0}}$ are $\mu$-almost-surely generic relative to $\tau_{G}$. For a generic $x\in G^{\mathbb{N}_0}$ relative to $\tau_{G}$, consider all generic  $y\in G^{\mathbb{N}_0}$ relative to $\tau_{G}$ that coincide with $x$ at all indices $i\geq1$. Now, for every $g\in G$, the associativity of $G$ implies that $m_g \circ \tau_G = \tau_G \circ m_g$ where $m_g$ is the multiplication by $g$ from the left. Thus $\left(m_g\right)_*\mu$ is also $\tau_G$-invariant. So in particular for such $y$, the measure $\left(m_{y_0x_0^{-1}}\right)_*\mu$ is $\tau_G$-invariant. Since $m_{y_0x_0^{-1}}$ is continuous, $y = m_{y_0x_0^{-1}} \left(x\right)$ must be $\left(m_{y_0x_0^{-1}}\right)_*\mu$-generic relative to $\tau_G$. However, it is $\mu$-generic relative to $\tau_G$, and hence $\left(m_{y_0x_0^{-1}}\right)_*\mu = \mu$. Denote by $H_{\mu,x}$ the set of all such values $y_{0}x_{0}^{-1}$.\\

By the associativity of $G$, the map $x\mapsto H_{\mu,x}$ (defined $\mu$-almost-surely) is $\tau_{G}$-invariant. By applying the ergodicity of $\mu$ relative to $\tau_{G}$ once again we deduce that there exists $H_\mu \subseteq G$ so that $H_{\mu,x}=H_\mu$. In particular, $H_\mu$ is thus closed under the group multiplication and hence it is a subgroup.
\end{proof}

\begin{remark}
As a corollary of Theorem \ref{pivato2}, we observe an that theorem enables us to attach a subgroup $H\leq G$ to any \underline{shift}-invariant Borel probability measure on $G^{\mathbb{N}_0}$ which is ergodic relative to the shift. The procedure is done by pushing the measure forward by the isomorphism between $\left(  G^{\mathbb{N}_0},\sigma \right)$ and $\left(  G^{\mathbb{N}_0},\tau_G \right)$ described in Remark \ref{biper} and then applying the theorem there. This $H$ is just the set of all elements $h\in G$ satisfying that the measure is preserved by the map $\left( x_0,x_1,x_2,\dots \right) \mapsto \left( hx_0,hx_1,hx_2,\dots \right)$ (this should not be confused with $K_\mu$ of Prop. \ref{HMM_gen1} - notice that here $h$ multiplies the values in all indices).

\end{remark}

When $\mu$ is also shift-invariant we will state and prove a generalized version of Theorem \ref{pivato2} that does not require ergodicity (again, a generalization with $H_{\mu,x}$ in place of $H_\mu$ is false without adding the shift-invariance assumption, and the proof fails because the ergodic decomposition relative to $\tau$ need not be measure-theoretically $\sigma^{-1}\mathcal{B}$-measurable). This is Theorem \ref{group_th}, and it will be proved with the aid of Theorem \ref{biper_th}.\\

In sake of completeness, we apply the method of the proof of Theorem \ref{pivato2} in the next subsection to prove a related theorem, interesting in its own right, but not required afterwards (except for its own improvement - Theorem \ref{HMM_gen}).

\subsection{A Slight Diversion: Ergodicity Relative to the Shift for Abelian Groups}

Allowing ourselves a slight diversion from the main thread of general finite groups, the following result is proved in \cite{key-8} by B. Host, A. Maass and S. Martinez.\\

\begin{proposition}
\label{HMM}

Let $p>0$ be a prime integer, and let $\mu$ be a Borel probability measure on $\left( \mathbb{Z}/p\mathbb{Z} \right)^{{\mathbb{N}_0}}$ which is invariant under both the shift and $\tau_{ \mathbb{Z}/p\mathbb{Z} }$ and ergodic already relative to the shift. If $h_\mu \left( \sigma \right) > 0$ then $\mu$ is the uniform measure.
\end{proposition}

\vspace{0.4cm}

Here the ergodicity assumption is relative to $\sigma$ (if instead it were relative to $\tau_{ \mathbb{Z}/p\mathbb{Z} }$ then this would have been an immediate conclusion from Theorem \ref{pivato2}). We can generalize this result by a proof similar to that of Theorem \ref{pivato2}. In particular, this generalization has no $\tau_{ \mathbb{Z}/p\mathbb{Z}}$-invariance assumption.\\

\begin{proposition}
\label{HMM_gen1}

Let $G$ be a finite \underline{Abelian} group, and let $\mu$ be a shift-invariant Borel probability measure on $G^{{\mathbb{N}_0}}$ that is ergodic relative to $\sigma ^2$. Then there exists a subgroup $K_{\mu} \leq G$ so that $\mu_{x}^{\tau_G^{-1}\mathcal{B}}$ ($\mathcal{B}$ being the Borel $\sigma$-algebra of $G^{\mathbb{N}_0}$) as a measure on $G$ (the atom of $x$ in $\tau_G^{-1}\mathcal{B}$ is determined by $x_0$) is $\mu$-almost-surely the uniform measure on a right $K_\mu$-coset.

\end{proposition}

\begin{proof}

Referring to the group operation as addition, and applying the notation of Remark \ref{biper} for the transpose cellular automaton: 
\[ \left(\hat{\tau_G}\left(y\right)\right)_i = y_{i+1}-y_i. \]

The topological dynamical system $\left(G^{\mathbb{N}_0},\tau_G,\sigma,\right)$ is isomorphic to $\left(G^{\mathbb{N}_0},\sigma,\hat{\tau_G} \right)$ by the map $\left(x_0,x_1,x_2,\dots\right)\mapsto \left(\left(\tau\left(x\right)\right)_0,\left(\tau\left(x\right)\right)_1,\left(\tau\left(x\right)\right)_2,\dots\right)$. Denote the push-forward of $\mu$ through this map by $\hat{\mu}$.\\

What we need to prove is equivalent to substituting in the statement of the proposition $\mu$ with $\hat{\mu}$, $\mu_{x}^{\tau_G^{-1}\mathcal{B}}$ with $\hat{\mu}^{\sigma^{-1}\mathcal{B}}_x$, $\sigma$ and $\tau_G$ with $\hat{\tau_G}$ and $\sigma$ respectively, and $K_{\mu}$ with $\hat{K}_{\hat{\mu}}$.\\

By the ergodicity of $\hat{\mu}$ relative to $\hat{\tau_G}^2$, the points in $G^{{\mathbb{N}_0}}$ are $\hat{\mu}$-almost-surely generic relative to $\hat{\tau_G}^2$. For such an $x\in G^{\mathbb{N}_0}$, consider all such $y\in G^{\mathbb{N}_0}$ that also coincide with $x$ at all indices $i\geq1$. Denote by $\hat{K}_{\hat{\mu},x}$ the set of all values $y_{0}-x_{0}$.
\\

By the associativity of $G$, the map $x\mapsto \hat{K}_{\hat{\mu},x}$ satisfies $  \hat{K}_{\hat{\mu},\hat{\tau_G}\left(x\right)} = -\hat{K}_{\hat{\mu},x}$ and is thus $\hat{\tau_G}^2$-invariant. The ergodicity of $\hat{\mu}$ relative to $\hat{\tau_G}^2$ implies that there exists $\hat{K}_{\hat{\mu}}\subseteq G$ so that $\hat{K}_{\hat{\mu},x} = \hat{K}_{\hat{\mu}}$  $\hat{\mu}$-almost-surely. In particular, $\hat{K}_{\hat{\mu}}$ must then be closed under the group addition and hence it is a subgroup.
\end{proof}

We do not know whether Prop.  \ref{HMM_gen1} remains true if the ergodicity relative to $\sigma^2$ is replaced with ergodicity just relative to $\sigma$. However, the current formulation suffices to draw the following conclusion.

\begin{corollary}
\label{hmm_gen2}
    Let $p>0$ be a prime integer, and let $\mu$ be a shift-invariant probability measure on $\left( \mathbb{Z}/p\mathbb{Z} \right)^{\mathbb{N}_0}$ which is ergodic relative to the shift. If  $-\int \log \mu_x^{\tau_{\mathbb{Z}/p\mathbb{Z}}^{-1}\mathcal{B}}\left(  \left[  x \right]_{\tau^{-1}_{\mathbb{Z}/p\mathbb{Z}}\mathcal{B}} \right)\,d\mu > 0$ then $\mu$ is the uniform measure.
\end{corollary}
\begin{proof}
    In the notation of the preceding proof, $  \hat{K}_{\hat{\mu},\hat{\tau_G}\left(x\right)} = -\hat{K}_{\hat{\mu},x}$ and thus $x\mapsto \left|\hat{K}_{\hat{\mu},x}\right|$ is $\hat{\tau_G}$-invariant and thus is $\hat{\mu}$-almost-surely equal to some constant $C$ (by the ergodicity of $\mu$ relative to $\sigma$). So 
    \[ \log C = \int \log C\,d\hat{\mu}  =-\int \log \mu_x^{\tau_{\mathbb{Z}/p\mathbb{Z}}^{-1}\mathcal{B}}\left(  \left[  x \right]_{\tau^{-1}_{\mathbb{Z}/p\mathbb{Z}}\mathcal{B}} \right)\,d\mu > 0,
    \]
    and hence $C>1$.\\
    
    Applying ergodic decomposition relative to $\sigma^2$, we know by Prop. \ref{HMM_gen1} that $C$ must be the size of some subgroup of  $\mathbb{Z}/p\mathbb{Z}$, and since $p$ is prime we deduce that $C=p$. Hence $\hat{\mu}$ is the uniform measure on $\left( \mathbb{Z}/p\mathbb{Z} \right)^{\mathbb{N}_0}$ and hence also is $\mu$.
\end{proof}

Prop. \ref{HMM} follows from Prop. \ref{hmm_gen2}, because in Prop. \ref{HMM}  $\mu$ is also $\tau_G$-invariant and thus $-\int \log \mu_x^{\tau_{\mathbb{Z}/p\mathbb{Z}}^{-1}\mathcal{B}}\left(  \left[  x \right]_{\tau^{-1}_{\mathbb{Z}/p\mathbb{Z}}\mathcal{B}} \right)\,d\mu = h_\mu\left(\tau_G\right)=h_\mu\left(\sigma\right)>0$.\\

For $\mu$ that  is also $\tau$-invariant, Prop. \ref{HMM_gen1} will be generalized to not contain any ergodicity assumption. This is Theorem \ref{HMM_gen}.

\section{A Basic Thereom for Bi-Permutative Cellular Automata}
\label{basic_theorem_sect}

This section (as well as the proof of Theorem \ref{group_th}) is based on the approach in which W. Parry proved Rudolph's theorem (cf. \cite{key-2}). Some of its presentation here is taken from the survey \cite{key-6}.\\

\subsection{Invariance of Conditional Probabilities for Certain Commuting Maps}

Let $\left(X,\mathcal{B},\nu\right)$ be a Standard Borel probability
space (i.e. standard Borel space with a probability measure), and
$S:X\rightarrow X$ be a measurable map that preserves $\nu$.\\

Any measurable function $w:X\rightarrow\left[0,\infty\right]$ induces a positive (maybe infinite) measure $\nu_{w}$ defined by setting
$d\nu_{w}=w\,d\nu$. The measure $\nu_{w}$ can be pushed-forward through $S$
to obtain a new measure on $X$ which is also absolutely continuous with respect to $\nu$, and let us denote its Radon-Nikodym derivative
by $L_{S}\left(w\right)$ ($L_{S}$ is called the transfer operator of $S$). One
may verify that
\[
\left(L_{S}\left(w\right)\right)\left(S\left(x\right)\right) = \mathbb{E}_\nu \left(w | S^{-1}\mathcal{B} \right) = \int w\left(y\right)\,d\nu_{x}^{S^{-1}\mathcal{B}}\left(y\right).
\]

(Notice that in particular $L_{S}\left(w\circ S\right) \left(S\left(x\right)\right)=w \left(S\left(x\right)\right)$, and because $S$ preserves $\mu$ this implies that
$ L_{S}\left(w\circ S\right)=w$ - but we shall not make use of this fact.)\\

We shall need the following relation between the transfer operator and the Koopman operator of $S$.\\

\begin{proposition}
\label{gpt}
    Let $\left(X,\mathcal{B},\nu\right)$ be a Standard Borel probability space, $S:X\rightarrow X$ a measurable map that preserves $\nu$ and $w\in L^1_\nu \left(X\right)$. If $L_S \left(w\right) = w$ then $w\circ S = w$. \footnote{This proposition is well known (see for example exercise 10.19 in \cite{eisner}). However, the author learned the specific proof presented here from ChatGPT, and neither he or it could find a reference for this exact proof. In addition, ChatGPT had also pointed out how to deduce from this proposition its truth for a measurable non-negative $w$ that may not be integrable, however we do not need here this strengthening.}
\end{proposition}

\begin{proof}
    Notice that
    \[
    w\circ S^n = L_{S}^n\left(w\right) \circ S^n = L_{S^n}\left(w\right) \circ S^n = 
    \mathbb{E}_\nu \left(w | S^{-n}\mathcal{B} \right).
    \]

    So by the reverse martingale convergence theorem, the sequence $ w\circ S^n$ converges in $L_\nu ^1\left(X\right)$. In particular, \[ \|  w\circ S^{n+1} - w\circ S^n \|_1 \xrightarrow[n\rightarrow \infty]{} 0. \] 
    However, since composition by $S$ is an isometry ($S$ is measure preserving), we have that \[ \|  w\circ S - w \|_1 \xrightarrow[n\rightarrow \infty]{} 0, \] 
    i.e. that \[ w \circ S = w.
    \]
\end{proof}

If $S$ is countable-to-one then, under the convention $0\cdot\infty=0$, we have that $\left(L_{S}w\right)\left(S\left(x\right)\right) = \underset{y\in S^{-1}\left(\{S\left(x\right)\}\right)}{\sum}\nu_{y}^{S^{-1}\mathcal{B}}\left(\left\{ y\right\} \right)\cdot w\left(y\right)$. Because $S$ preserves $\nu$, this implies that
\[\left(L_{S}w\right)\left(x\right) = \underset{y\in S^{-1}\left(\{x\}\right)}{\sum}\nu_{y}^{S^{-1}\mathcal{B}}\left(\left\{ y\right\} \right)\cdot w\left(y\right).\]\\

Parry proved the next lemma. Its proof is included for the convenience of the reader.\\

\begin{lemma}
\label{parry_lemma}
    Let $\left(X,\mathcal{B},\nu\right)$ be a Standard
Borel probability space and $S:X\rightarrow X$ a countable-to-one
measurable map that preserves $\nu$. If $T:X\rightarrow X$ is
measurable map satisfying for some $M>0$ that $\left|T^{-1}\left(\left\{ x\right\} \right)\right|\leq M$ for all $x\in X$ \footnote{The lemma remains true if also $T$ is just countable-to-one. However, this requires the strengthened version of Prop. \ref{gpt} mentioned in the previous footnote which we chose not to present.}, preserves $\nu$, commutes with $S$ and for $\nu$-almost-every
$x$ the map $T|_{S^{-1}\left(\{x\}\right)}:S^{-1}\left(\left\{ x\right\} \right)\rightarrow S^{-1}\left(\left\{ T\left(x\right)\right\} \right)$
is bijective, then $\nu_{T\left(x\right)}^{S^{-1}\mathcal{B}}\left(\left\{ T\left(x\right)\right\} \right)=\nu_{x}^{S^{-1}\mathcal{B}}\left(\left\{ x\right\} \right)$
and $\nu_{S\left(x\right)}^{T^{-1}\mathcal{B}}\left(\left\{ S\left(x\right)\right\} \right)=\nu_{x}^{T^{-1}\mathcal{B}}\left(\left\{ x\right\} \right)$ $\nu$-almost-surely.\\

\end{lemma}

\begin{proof}

The commutation relation implies that $\nu_{y}^{\left(TS\right)^{-1}\mathcal{B}}\left(\left\{ y\right\} \right)=\nu_{y}^{\left(ST\right)^{-1}\mathcal{B}}\left(\left\{ y\right\} \right)$
and hence the identity 

\begin{equation}    
\label{**}
\nu_{y}^{S^{-1}\mathcal{B}}\left(\left\{ y\right\} \right)\cdot\nu_{S\left(y\right)}^{T^{-1}\mathcal{B}}\left(\left\{ S\left(y\right)\right\} \right)=\nu_{y}^{T^{-1}\mathcal{B}}\left(\left\{ y\right\} \right)\cdot\nu_{T\left(y\right)}^{S^{-1}\mathcal{B}}\left(\left\{ T\left(y\right)\right\} \right).
\end{equation}

Now, $\left(L_{S}\frac{1}{\nu_{\left(\cdot\right)}^{T^{-1}\mathcal{B}}\left(\left\{ \left(\cdot\right)\right\} \right)}\right)\left(x\right) = \underset{y\in S^{-1}\left(\{x\}\right)}{\sum}\frac{\nu_{y}^{S^{-1}\mathcal{B}}\left(\left\{ y\right\} \right)}{\nu_{y}^{T^{-1}\mathcal{B}}\left(\left\{ y\right\} \right)}$,
and by \eqref{**} this is equal to

\[
\underset{y\in S^{-1}\left(\{x\}\right)}{\sum}\frac{\nu_{T\left(y\right)}^{S^{-1}\mathcal{B}}\left(\left\{ T\left(y\right)\right\} \right)}{\nu_{S\left(y\right)}^{T^{-1}\mathcal{B}}\left(\left\{ S\left(y\right) \right\} \right)} = \underset{y\in S^{-1}\left(\{x\}\right)}{\sum}\frac{\nu_{T\left(y\right)}^{S^{-1}\mathcal{B}}\left(\left\{ T\left(y\right)\right\} \right)}{\nu_{x}^{T^{-1}\mathcal{B}}\left(\left\{ x \right\} \right)} = \frac{\underset{y\in S^{-1}\left(\{x\}\right)}{\sum}\nu_{T\left(y\right)}^{S^{-1}\mathcal{B}}\left(\left\{ T\left(y\right)\right\} \right)}{\nu_{x}^{T^{-1}\mathcal{B}}\left(\left\{ x \right\} \right)}
\]

\[
=\frac{1}{\nu_{x}^{T^{-1}\mathcal{B}}\left(\left\{ x \right\} \right)},
\]

where in the last equality we applied the bijectivity assumption.\\

Hence $\frac{1}{\nu_{x}^{T^{-1}\mathcal{B}}\left(\left\{ x\right\} \right)}$ is fixed by $L_{S}$. Since

\[\int \frac{1}{\nu_{x}^{T^{-1}\mathcal{B}}\left(\left\{ x\right\} \right)} d\nu\left(x\right) = \int \int \frac{1}{\nu_{y}^{T^{-1}\mathcal{B}}\left(\left\{ y\right\} \right)} d\nu_{x}^{T^{-1}\mathcal{B}}\left(y\right) d\nu\left(x\right) \leq \int M \, d\nu\left(x\right) =  M, \] 
the function $\frac{1}{\nu_{x}^{T^{-1}\mathcal{B}}\left(\left\{ x\right\} \right)}$ belongs to $L_\nu ^1 \left(X\right)$. Hence by Prop. \ref{gpt} also
$\frac{1}{\nu_{S\left(x\right)}^{T^{-1}\mathcal{B}}\left(\left\{ S\left(x\right)\right\} \right)}=\frac{1}{\nu_{x}^{T^{-1}\mathcal{B}}\left(\left\{ x\right\} \right)}$. Hence $\nu_{S\left(x\right)}^{T^{-1}\mathcal{B}}\left(\left\{ S\left(x\right)\right\} \right)=\nu_{x}^{T^{-1}\mathcal{B}}\left(\left\{ x\right\} \right)$ and by \eqref{**} also $\nu_{T\left(x\right)}^{S^{-1}\mathcal{B}}\left(\left\{ T\left(x\right)\right\} \right)=\nu_{x}^{S^{-1}\mathcal{B}}\left(\left\{ x\right\} \right)$. 
\end{proof}

\begin{remark}
\label{parry_remark}

In fact, the proof will still hold if the bijectivity requirement on  $T|_{S^{-1}x}:S^{-1}\left(\left\{ x\right\} \right)\rightarrow S^{-1}\left(\left\{ T\left(x\right)\right\} \right)$ is loosened up to $\nu$-almost-sure bijectivity of the restiction of this map from $\left\{ y\in S^{-1}\left(\left\{ x\right\} \right)\,:\,\nu_{y}^{S^{-1}\mathcal{B}}\left(\left\{ y\right\} \right)>0\right\} $ to $\left\{ y\in S^{-1}\left(\left\{ T\left(x\right)\right\} \right)\,:\,\nu_{y}^{S^{-1}\mathcal{B}}\left(\left\{ y\right\} \right)>0\right\}$.\\

\end{remark}

\subsection{Recalling the Pinsker $\sigma$-algebra of a measure preserving system}

Let $\left(X,\mathcal{B},\nu\right)$ be a standard Borel probability space and $T:X\rightarrow X$ a measure preserving map. \textit{The Pinsker $\sigma$-algebra} of the system $\left(X,\mathcal{B},\nu,T\right)$ is defined as the collection of all sets $A\in \mathcal{B}$ for which $h\left(\{A,X\setminus A\},T\right) = 0$.\\

As the Pinsker $\sigma$-algebra remains unchanged when applying $T^{-1}$ on it, it corresponds to a factor of  $\left(X,\mathcal{B},\nu,T\right)$ called \textit{the Pinsker factor} - and let us denote the factor map by $\varphi$. It can be verified from the definition of the Pinsker $\sigma$-algebra that the Pinsker factor  has zero entropy. Moreover, the Pinsker factor can be characterized by the following property: if  $\psi$ is a factor map from $X$ to a factor of zero entropy, there exists a unique map $\psi'$ from the Pinsker factor to that factor so that $\psi = \psi' \circ \varphi$. It is thus said that the Pinsker factor is "the maximal factor of zero entropy".\\

If $\xi$ is a partition satisfying $\vee_{i=0}^\infty T^{-i}\xi\stackrel{\mod\nu}{=}\mathcal{B}$ - or just $\vee_{i=-\infty}^\infty T^{-i}\xi\stackrel{\mod\nu}{=}\mathcal{B}$ when $T$ is invertible - then the Pinsker $\sigma$-algebra of $\left(X,\mathcal{B},\nu,T\right)$ is equal modulo $\mu$ to $\cap_{n=1}^\infty \vee_{i=n}^\infty T^{-i}\xi$. The latter is called \textit{the tails $\sigma$-algebra} (relative to the partition $\xi$).

\subsection{The Theorem}

We begin with two lemmas.\\

Before the first let, us recall the Abramov-Rokhlin formula: given an invertible measure preserving map $T:X \rightarrow X$ for a standard Borel probability space $\left(X,\mathcal{B},\nu\right)$, and a sub-$\sigma$-algebra $\mathcal{A}\subseteq \mathcal{B}$ satisfying $T^{-1}\mathcal{A} \stackrel{\mod \nu}{=} \mathcal{A}$, the formula says that $h_\nu\left(T\right)$ equals $h_\nu\left(T\,|\,\mathcal{A}\right)$ plus the entropy of the factor whose associated sub-$\sigma$-algebra in $\mathcal{B}$ is $\mathcal{A}$.

\begin{lemma}
\label{pinsker_lemma1}
Assume $T:X \rightarrow X$ is a measure preserving map
for a standard Borel probability space $\left(X,\mathcal{B},\nu\right)$, and that there
exists a finite partition $\xi$ (or countable of finite entropy) for which $\vee_{i=0}^\infty T^{-i}\xi\stackrel{\mod\nu}{=}\mathcal{B}$.
Then every $T$-invariant sub-$\sigma$-algebra $\mathcal{A}$ - i.e.
$T^{-1}\mathcal{A}=\mathcal{A}$ - is contained in the Pinsker $\sigma$-algebra
of the system. \end{lemma}

\begin{proof}
Note that
\[
h_{\nu}\left(T\,|\,\mathcal{A}\right) = H_{\nu}\left(\mathcal{\xi}\,|\,\left(\vee_{i=1}^\infty T^{-i}\xi \right) \vee\mathcal{A}\right)
\]

\[ = H_{\nu}\left(\mathcal{\xi}\,|\,T^{-1}\mathcal{B}\vee T^{-1}\mathcal{A}\right)=H_{\nu}\left(\mathcal{\xi}\,|\,T^{-1}\mathcal{B}\right)=H_{\nu}\left(\mathcal{\xi}\,|\,\vee_{i=1}^\infty T^{-i}\xi \right)=h_{\nu}\left(T\right). \]
The result follows by passing to the natural invertible extension, and applying the Abramov-Rokhlin formula (see also the first paragraph in the proof of Lemma \ref{pinsker_lemma2} - it refers to a different invertible extension but the explanation given there is similar).
\end{proof}

\begin{lemma}
\label{pinsker_lemma2}
Let $\tau : \Lambda^{\mathbb{N}_0} \rightarrow\ \Lambda^{\mathbb{N}_0}$ be a right permutative cellular automaton. If $\mu$ is a Borel probability measure on $\Lambda^{\mathbb{N}_0}$   invariant under $\sigma$ and $\tau$ then the Pinsker $\sigma$-algebra of $\sigma$
is equal to that of $\tau$. 
\end{lemma}

\begin{proof}

First, let us note that it suffices to prove the required equality in the natural invertible extension of the system, that is in $\left\{ x\in\Lambda^{\mathbb{Z}^2}\,:\,\forall i\in\mathbb{Z}~~  x_{\left(\cdot\right),i+1}=\tau\left(x_{\left(\cdot\right),i}\right)\right\} $. This is because the pre-image of each of  the Pinsker $\sigma$-algebras of the original system through the factor map equals to the collection of sets in the Pinsker $\sigma$-algebra of the extension that are measurable with respect to the factor. So we work at the extension system.
\\

Denote by $\xi$ the partition according to the value of the points at index zero. Now, $\vee_{i=0}^{N-1}\sigma_{\left(1,0\right)}^{-i}\xi = \vee_{i=0}^{N-1}\sigma_{\left(0,1\right)}^{-i}\xi$ by the right permutativity of $\tau$, and hence $h_\mu\left(\sigma_{\left(1,0\right)}\,|\,\mathcal{A}\right) = h_\mu\left(\sigma_{\left(0,1\right)}\,|\,\mathcal{A}\right)$   for every sub-$\sigma$-algebra $\mathcal{A}$ of the Borel $\sigma$-algebra which satisfies $\sigma_{\left(1,0\right)}^{-1}\mathcal{A} = \sigma_{\left(0,1\right)}^{-1}\mathcal{A} = \mathcal{A}$ (mod $\mu$).\\

Denote by $\mathcal{P}_{\sigma_{\left(1,0\right)}}$ and $\mathcal{P}_{\sigma_{\left(0,1\right)}}$ the Pinsker $\sigma$-algebras of the corresponding maps. Not only $\sigma_{\left(1,0\right)}^{-1}\mathcal{P}_{\sigma_{\left(1,0\right)}} = \mathcal{P}_{\sigma_{\left(1,0\right)}}$, but also 
$\sigma_{\left(0,1\right)}^{-1}\mathcal{P}_{\sigma_{\left(1,0\right)}} = \mathcal{P}_{\sigma_{\left(1,0\right)}}$ since $\sigma_{\left(0,1\right)}$ is an automorphism of the system of the map $\sigma_{\left(1,0\right)}$. Hence $h_\mu\left(\sigma_{\left(1,0\right)}\,|\,\mathcal{P}_{\sigma_{\left(1,0\right)}}\right) = h_\mu\left(\sigma_{\left(0,1\right)}\,|\,\mathcal{P}_{\sigma_{\left(1,0\right)}}\right)$. But \[h_\mu\left(\sigma_{\left(1,0\right)}\,|\,\mathcal{P}_{\sigma_{\left(1,0\right)}}\right) = h_\mu\left(\sigma_{\left(1,0\right)}\right) = h_\mu\left(\sigma_{\left(0,1\right)}\right),\]
so $h_\mu\left(\sigma_{\left(0,1\right)}\right) = h_\mu\left(\sigma_{\left(0,1\right)}\,|\,\mathcal{P}_{\sigma_{\left(1,0\right)}}\right)$. By the Abramov-Rokhlin formula this implies that $\mathcal{P}_{\sigma_{\left(1,0\right)}}\subseteq\mathcal{P}_{\sigma_{\left(0,1\right)}}$. The reverse containment is proved similarly.
\end{proof}

\vspace{0.5cm}

We are now in position to prove our improved version of Theorem \ref{pivato1}.

\vspace{0.5cm}

\begin{theorem}
\label{biper_th}

Let $\tau:\Lambda^{{\mathbb{N}_0}}\rightarrow\Lambda^{{\mathbb{N}_0}}$
be a bi-permutative cellular automaton, and let $\mu$ be a shift-invariant
Borel probability measure on $\Lambda^{{\mathbb{N}_0}}$ which is also
invariant relative to $\tau$. Then the function $x \mapsto \mu _x ^{\sigma^{-1}\mathcal{B}}\left(\{x\}\right)$ (where $\mathcal{B}$ is the Borel $\sigma$-algebra of $\Lambda^{{\mathbb{N}_0}}$) is $\tau$-invariant, measurable with respect to the Pinsker $\sigma$-algebras of both maps (which are equal), and - when considered as a probability measure on $\Lambda$  - $\mu _x ^{\sigma^{-1}\mathcal{B}}$ is the uniform probability on a subset $\Lambda_{\mu,x} \subseteq \Lambda$.
\end{theorem}

\begin{proof}

    The left permutativity of $\tau$ guarantees that the conditions of Lemma \ref{parry_lemma} are fulfilled for $\sigma, \tau$ in place of $S,T$ respectively. Thus  $x \mapsto \mu _x ^{\sigma^{-1}\mathcal{B}}\left(\{x\}\right)$ is $\tau$-invariant.\\

    By Lemma \ref{pinsker_lemma1}, the function $x \mapsto \mu _x ^{\sigma^{-1}\mathcal{B}}\left(\{x\}\right)$ is measurable relative to the Pinsker $\sigma$-algebra of $\tau$. By this and by the right permutativity of $\tau$, applying Lemma \ref{pinsker_lemma2}, $x \mapsto \mu _x ^{\sigma^{-1}\mathcal{B}}\left(\{x\}\right)$ is also measurable relative to the Pinsker $\sigma$-algebra of $\sigma$ since it is the same.\\

    The last claim follows from the fact that measurability relative the Pinsker $\sigma$-algebra of $\sigma$ implies, in particular, measurability relative to $\sigma^{-1}\mathcal{B}$, i.e., points in the same atom of $\sigma^{-1}\mathcal{B}$ have the same value of $x \mapsto \mu _x ^{\sigma^{-1}\mathcal{B}}\left(\{x\}\right)$ $\mu$-almost-surely.  
\end{proof}

\vspace{0.3cm}

\begin{remark}
\label{biper_th_rem}
    By Remark \ref{biper} regarding the transpose cellular automaton, the roles of $\sigma$ and $\tau$ in the consequent of Theorem \ref{biper_th} may be reversed.
\end{remark}

So in the setting of Theorem \ref{biper_th}, it adds on Theorem \ref{pivato1} in that that it can be formulated with conditioning the probability only on values along the indices of the point in the system $\Lambda^{{\mathbb{N}_0}}$ - and not on the two-dimensional subshift which appears in the latter. Moreover, the formulation does not require any form of ergodicity.\\

An illustrative and simple well-known example is the following.\\

\begin{example}
\label{kitchens}
    Consider the Ledrappier cellular automaton (the case of the multiplication cellular automaton $\tau_G$ for $G$ which is
the additive group $\mathbb{Z}/2\mathbb{Z}$).\\

A well-known example for $\mu$ that is not the uniform measure is
\[
\mu = \frac{1}{4}m_\text{odd}+ \sigma_* \frac{1}{4}m_\text{odd}+ \frac{1}{4}\left(\tau_G\right)_*m_\text{odd}+ \frac{1}{4}\left(\tau_G\right)_*\sigma_*m_\text{odd},
\]
where $m_{odd}$ is the push-forward of the uniform measure through the map $\left(x_{0},x_{1},x_{2},\dots\right)\mapsto\left(0,x_{0},0,x_{1},0,x_{2},0,\dots\right)$. $\mu$ is obviously shift-invariant. To see that it is $\tau_G$-invariant
one can verify that $\tau_G^2 = id + \sigma^2$. So $\left(\tau_G\right)_{*}^{2}m_{odd}=m_{odd}$ and $\left(\tau_G\right)_{*}^{2}\sigma_{*}m_{odd}=\sigma_{*}m_{odd}$.\\

In fact, $\mu$ is ergodic relative to the joint action of $\sigma$ and $\tau_G$. Te measures $m_\text{odd}, \sigma_* m_\text{odd},\left(\tau_G\right)_*m_\text{odd},\left(\tau_G\right)_*\sigma_*m_\text{odd}$ are mutually singular, namely there exist disjoint
Borel measurable sets $X_{1},X_{2},X_{3},X_{4}\subseteq\left(\mathbb{Z}/2\mathbb{Z}\right)^{{\mathbb{N}_0}}$ of probability $1$ with respect to the corresponding measures. Consider a Borel measurable set $A\subseteq\left(\mathbb{Z}/2\mathbb{Z}\right)^{{\mathbb{N}_0}}$ invariant under both maps. The disjoint sets $A\cap X_{1},A\cap X_{2},A\cap X_{3},A\cap X_{4}$ are transitively permuted by the action of $\sigma^{-1},\tau_G^{-1}$ and hence
all have equal probability. But invariance of $A$ under $\sigma^{-2}$ implies that the set $A\cap X_{1}$ is of $m_{odd}$
probability either $0$ or $1$. Hence $\mu\left(A\right)$ is equal to either $0$ or $1$.\\

This example demonstrates, that even when $\mu$ is ergodic relative
to the joint action of $\sigma$ and $\tau_G$, the set $\Lambda_{\mu,x}$ in Theorem \ref{biper_th} may
indeed depend on the point (in this example it is either $\left\{ 0\right\} $
or $\mathbb{Z}/2\mathbb{Z}$). As always, $\left( \pi_{\mu} \right)_{*}\mu$ is shift-invariant,
but the fibers of $\pi_{\mu}$ are not measure-theoretically respected by
$\tau_G$. Moreover, $h_{\left( \pi_{\mu} \right)_{*}\mu}\left(\sigma\right)>0$. Section \ref{zero_ent_section} will deal with the special class of invariant measures that satisfy the requirements of Theorem \ref{biper_th} but for which all these troublesome phenomena do not occur. This will be shown to hold exactly when $h_{\left(\pi_\mu\right)_*\mu}\left(\sigma\right)=0$. In the present section we shall see that this condition is sometimes guaranteed to hold.
\end{example}

\subsection{The Structural Factor $\pi_{\mu} : \Lambda^{\mathbb{N}_0} \rightarrow \left(2^\Lambda\right)^{\mathbb{N}_0}$}

\label{pi_subsect}

We denote $2^{\Lambda} = \{A\,:\, A\subseteq \Lambda \} $.\\

The map $\pi_{\mu}$  that we now define will serve us throughout this paper.\\

\begin{definition}
Given a bi-permutative cellular automaton $\tau:\Lambda^{{\mathbb{N}_0}}\rightarrow\Lambda^{{\mathbb{N}_0}}$
 and a shift-invariant Borel probability measure  $\mu$ on $\Lambda^{\mathbb{N}_0}$ which is also
invariant relative to $\tau$, denote $\Lambda_{\mu ,x}$ as in Theorem \ref{biper_th}. We define  $\pi_{\mu} : \Lambda^{\mathbb{N}_0} \rightarrow \left(2^\Lambda\right)^{\mathbb{N}_0}$ as the shift-equivariant map defined by the requirement $\left( \pi_{\mu} \left( x \right) \right)_0 = \Lambda_{\mu,x}$. We call it \textit{the structural factor}.
\end{definition}

The measurablity of the map $\pi_\mu$ follows from the measurability of the conditional probability map $x\mapsto \mu_x^{\sigma^{-1}\mathcal{B}}$.\\

We shall later need the following elementary lemma concerning the entropy of  the system $\left(\left(2^{\Lambda}\right)^{{\mathbb{N}_0}},\sigma,\left( \pi_{\mu} \right)_{*}\mu\right)$.\\

\begin{lemma}
\label{zero_ent}
In the setting of Theorem \ref{biper_th}, $h_{\left( \pi_{\mu} \right)_{*}\mu}\left(\sigma\right)=0$ if and only if
$\mu$-almost-surely for all $1\leq k\in{\mathbb{N}_0}$ the subset $\Lambda_{\mu,x}$ is $\sigma^{-k}\mathcal{B}$-measurable.

\end{lemma}

\begin{proof}
$\left(\rightarrow\right)$ is immediate, for then $\sigma^{-1}\mathcal{B} = \mathcal{B}$. As for $\left(\leftarrow\right)$, the assumption says that for any $A\subseteq\Lambda$ the event $\left\{ x\,:\,\Lambda_{\mu,x}=A\right\} $ is a tail event and belongs to the Pinsker $\sigma$-algebra of $\left(\Lambda^{\mathbb{N}_0},\sigma,\mu\right)$.
The partition of $\left(2^{\Lambda}\right)^{{\mathbb{N}_0}}$ with respect
to the value at index zero is composed of events who under $\pi_{\mu}^{-1}$ are equal to the events of the form $\left\{ x:\,\Lambda_{\mu,x}=A\right\}$, and hence the structural factor $\left(\left(2^{\Lambda}\right)^{{\mathbb{N}_0}},\sigma,\left( \pi_{\mu} \right)_{*}\mu\right)$
of $\left(\Lambda^{\mathbb{N}_0},\sigma,\mu\right)$ is measurable relative to the Pinsker $\sigma$-algebra of  $\left(\Lambda^{\mathbb{N}_0},\sigma,\mu\right)$, so it actually factors
through the Pinsker factor of $\left(\Lambda^{\mathbb{N}_0},\sigma,\mu\right)$, and thus $h_{\left( \pi_{\mu} \right)_{*}\mu}\left(\sigma\right)=0$.
\end{proof}

Examining again Example \ref{kitchens}, $\left( \pi_{\mu} \right)_{*}\mu$ is shift-invariant,
but the fibers of $\pi_{\mu}$ are not measure-theoretically respected by
$\tau_G$. Moreover, $h_{\left( \pi_{\mu} \right)_{*}\mu}\left(\sigma\right)>0$. Section \ref{zero_ent_section} will deal with the special class of invariant measures that satisfy the requirements of Theorem \ref{biper_th} but for which all these troublesome phenomena do not occur. This will be shown to hold exactly when $h_{\left(\pi_\mu\right)_*\mu}\left(\sigma\right)=0$. In the next section we shall see that this condition is sometimes guaranteed to hold.

\section{Group Multiplication Cellular Automata}
\label{group_sect}

As promised, we are now able to generalize Theorem \ref{pivato2} beyond ergodicity for the special case in which $\mu$ is also shift-invariant.\\

\begin{theorem}
\label{group_th}
Let $G$ be a finite group, and let $\mu$ be a shift-invariant Borel probability measure on $G^{\mathbb{N}_0}$ that is also invariant relative to $\tau_G$. Then
$\mu_{x}^{\sigma^{-1}\mathcal{B}}$ ($\mathcal{B}$  being the Borel $\sigma$-algebra of $G^{\mathbb{N}_0}$) as a measure on $G$ is $\mu$-almost-surely
the uniform measure on a right coset of some subgroup $H_{\mu, x}\leq G$. The
subgroup $H_{\mu, x}$ may depend on $x$ but is invariant relative
to $\tau$ and also measurable relative to the Pinsker $\sigma$-algebras of both maps (which are equal).

\end{theorem}

\begin{proof}

Since $\tau_G$ is bi-permutative, Theorem \ref{biper_th} applies
here. The $\tau$-invariance of $\mu_{x}^{\sigma^{-1}\mathcal{B}}\left(\left\{ x\right\} \right)$
together with the associativity of $G$ imply the $\tau$-invariance
of the set $H_{\mu,x}$ defined $\mu$-almost-surely as the set of elements $g$ of $G$ for
which the point $\left(\left(gx_{0}\right),x_{1},x_{2},\dots\right)$ satisfies $\mu_{x}^{\sigma^{-1}\mathcal{B}}\left(\left\{ \left(\left(gx_{0}\right),x_{1},x_{2},\dots\right)\right\} \right)=\mu_{x}^{\sigma^{-1}\mathcal{B}}\left(\left\{ x\right\} \right)$.\\

In particular, this means that $x\mapsto H_{\mu,x}$ is measurable with
respect to the Pinsker $\sigma$-algebra of $\tau$ and hence of $\sigma$.
So it is $\sigma^{-1}\mathcal{B}$-measurable (i.e., points in the same atom of $\sigma^{-1}\mathcal{B}$ have the same value of $x \mapsto H_{\mu,x}$ $\mu$-almost-surely), and hence $H_{\mu,x}$ as a set is closed under the group multiplication operation and hence is a subgroup of $G$.
\end{proof}

Notice that Example \ref{kitchens} serves also as a simple example for Theorem \ref{group_th}.\\

\subsection{A Dual Version to Theorem \ref{group_th} which Holds for Abelian Groups}

We now state and prove a generalization of Prop. \ref{HMM_gen1} when $\mu$ is also $\tau$-invariant. The ergodicity assumption of the that proposition is omitted.

\begin{proposition}
\label{HMM_gen}

Let $G$ be an \underline{Abelian} group, and let $\mu$ be a shift-invariant Borel probability measure on $G^{\mathbb{N}_0}$ that is also invariant relative to $\tau_G$. Then $\mu_{x}^{\tau_G^{-1}\mathcal{B}}$ ($\mathcal{B}$  being the Borel $\sigma$-algebra of $G^{\mathbb{N}_0}$) as a measure on $G$ is $\mu$-almost-surely the uniform measure on a coset of some subgroup $K_{\mu, x}\leq G$. The subgroup $K_{\mu, x}$ may depend on $x$ but is invariant relative to $\sigma$ and also measurable relative to the Pinsker $\sigma$-algebras of both maps.
\end{proposition}

\begin{proof}

Referring to the group operation as addition, and applying the notation of Remark \ref{biper} for the transpose cellular automaton: 
\[ \left(\hat{\tau_G}\left(y\right)\right)_i = y_{i+1}-y_i. \]

The topological dynamical system $\left(G^{\mathbb{N}_0},\tau _G,\sigma,\right)$ is isomorphic to $\left(G^{\mathbb{N}_0},\sigma,\hat{\tau_G} \right)$ by the map $\left(x_0,x_1,x_2,\dots\right)\mapsto \left(\left(\tau_G\left(x\right)\right)_0,\left(\tau_G\left(x\right)\right)_1,\left(\tau_G\left(x\right)\right)_2,\dots\right)$. Denote the push-forward of $\mu$ through this map by $\hat{\mu}$.\\

Therefore, what we need to prove is equivalent to substituting in the statement of the proposition $\mu$ with $\hat{\mu}$, $\mu_{x}^{\tau_G^{-1}\mathcal{B}}$ with $\hat{\mu}^{\sigma^{-1}\mathcal{B}}_y$, and $\sigma$ and $\tau_G$ with $\hat{\tau_G}$ and $\sigma$ respectively.\\

So we apply Theorem \ref{biper_th} to the system of $\sigma,\hat{\tau_G}$ and $\hat{\mu}$. The $\hat{\tau_G}$-invariance of $\hat{\mu}_{y}^{\sigma^{-1}\mathcal{B}}\left(\left\{ y\right\} \right)$ - together with the associativity of $G$ and the fact it is Abelian - imply that the set $K_{\mu,y}$, defined $\hat{\mu}$-almost-surely as the set of elements $g \in G$ for which the point $\left(\left(gx_{0}\right),x_{1},x_{2},\dots\right)$ satisfies $\hat{\mu}_{y}^{\sigma^{-1}\mathcal{B}}\left(\left\{ \left(\left(gy_{0}\right),y_{1},y_{2},\dots\right)\right\} \right)=\hat{\mu}_{y}^{\sigma^{-1}\mathcal{B}}\left(\left\{ y\right\} \right)$, is transformed to $-K_{\mu,y}$ under $\hat{\tau_G}$.\\

$K_{\mu,y}$ is thus $\hat{\tau_G}^2$-invariant. This means that $y\mapsto K_{\mu,y}$ is measurable with
respect to the Pinsker $\sigma$-algebra of $\tau_G^2$, hence of $\tau_G$ and hence of $\sigma$ (they are all equal).
So it is $\sigma^{-1}\mathcal{B}$-measurable, and hence $K_{\mu,x}$ as a set is closed under the group addition and hence is a subgroup of $G$.\\

It is left to prove that $K_{\mu,y}$ is in fact $\hat{\tau_G}$-invariant, where it was a priori known that $\hat{\mu}$-almost surely $K_{\mu,\hat{\tau}\left(y\right)} = -K_{\mu,y}$. But for a subgroup $-K_{\mu,y}= K_{\mu,y}$.
\end{proof}

\vspace{0.5cm}

\subsection{Back to the General Case}

Before continuing to state the next proposition, it may be a good time to notice that in the setting of Theorem \ref{pivato2} when $\mu$ is ergodic relative to $\tau_G$ - knowing $H_\mu$ - $\pi_{\mu}$ can just be defined algebraically as the map that applies component-wise the
canonical projection $Gg\mapsto Hg$. This is also the case even without ergodicity if it happens to be that there
is some $H\leq G$ for which $H_{\mu,x}=H$ $\mu$-almost-surely.\\

\begin{remark}
    As Example \ref{kitchens} demonstrates $h_{\left( \pi_\mu \right)_*\mu}\left(\sigma\right)$ does not necessarily vanish. Applying the Abramov-Rokhlin formula in an attempt to prove so would be erroneous since it presupposes that the $\sigma$-algebra in the original system corresponding to its structural factor $\pi_\mu$ is shift-invariant, i.e. that the shift map of the structural factor system is measure-theoretically invertible (which is what is aimed to prove in the first place).\\
\end{remark}

Prop. \ref{zero_ent_suff}, follows the next proposition, will form a sufficient condition for $h_{\left( \pi_\mu \right)_*\mu}\left(\sigma\right)$ to be equal to zero.

\begin{proposition}
\label{coset_stab}

Let $G$ be a finite group, and let $\mu$ be a $\tau_G$-invariant Borel probability measure on $G^{\mathbb{N}_0}$. We have two similar results:

\textbf{(i)} Suppose that $\mu$ is ergodic relative to $\tau_G$ and that $H_\mu$ (as in Theorem \ref{pivato2}) is a normal subgroup of
$G$. Then there exists a subgroup $U_\mu \leq H_\mu$, so that for $x\in G^{{\mathbb{N}_0}}$
$\mu$-almost-surely the partition of the elements of the coset $H_\mu x_{1}$ at index
$1$ into elements that determine the same coset of $H_\mu$ at index $0$ is exactly the partition into the cosets
of $U_\mu$ in $G$ that are contained in $H_\mu x_{1}$.

\textbf{(ii)} Suppose that $\mu$ is also shift-invariant and that there exists $H\vartriangleleft G$ so that $\mu$-almost-surely $H_{\mu,x}$ (as in Theorem \ref{group_th}) is equal to it. Then for $x\in G^{{\mathbb{N}_0}}$
there $\mu$-almost surely exists $U_{\mu,x}\leq H$
so that the partition of the elements of the coset $Hx_{1}$
at index $1$ into elements that determine the same coset of $H$ at index $0$ is exactly the partition into the cosets of $U_{\mu,x}$ in $G$ that are contained in $Hx_{1}$.
    
\end{proposition}

\begin{proof}

\textbf{(i)} Assume $\left(\,\cdot\,,y_{1},x_{2},x_{3},x_{4},\dots\right)$
determines the same coset $H_\mu x_{0}$ at index $0$ for some $y_{1}\in H_\mu x_{1}$.
Then, since $\mu_{\left(\cdot\right)}^{\sigma^{-1}\mathcal{B}}\left(\left\{ \left(\cdot\right)\right\} \right)$
is $\tau_G$-invariant, applying $\tau_G$ on $\left(\,x_{0}\,,x_{1},x_{2},x_{3},x_{4},\dots\right)$
and $\left(\,x_{0}\,,y_{1},x_{2},x_{3},x_{4},\dots\right)$ yields
that the two outcomes determine $H_\mu x_{0}x_{1}$ and $H_\mu x_{0}y_{1}$
respectively. The normality assumption implies that \[H_\mu x_{0}x_{1}=H_\mu x_{0} H_\mu x_{1}=H_\mu x_{0} H_\mu y_{1} = H_\mu x_{0} y_{1},\]
so also the two outcomes determine the same coset at index $0$.\\

Now consider the set \[ U_{\mu,x}=\left\{ h\in H_\mu\,:\,\left(\,\cdot\,,hx_{1},x_{2},x_{3},x_{4},\dots\right)\, \mbox {determines the coset} \,H_\mu x_{0}\right\}. \]
The last paragraph and the associativity of $G$ imply that the set-valued function $x\mapsto U_{\mu,x}$
is $\tau_G$-invariant. By the ergodicity assumption there is a set
$U_\mu \subseteq H_\mu $ so that $U_{\mu,x}=U_\mu$ $\mu$-almost surely. By the definition
of $U_{\mu,x}$, this implies that $U_\mu$ is closed under the group multiplication
operation and hence a subgroup of the finite group $H_\mu $.\\

\textbf{(ii)} The difference from the previous proof is that now the
$\tau_G$-invariant set-valued function $x\mapsto U_{\mu,x}$ need not be
constant. However, by its $\tau_G$-invariance, it is measurable relative
to the Pinsker $\sigma$-algebra of $\tau_G$ (by Lemma \ref{pinsker_lemma1}) and hence of $\sigma$
(by Lemma \ref{pinsker_lemma2}). Hence it is measurable relative to $\sigma^{-2}\mathcal{B}$. Thus $\mu$-almost-surely  $U_{\mu,x}$ is closed under the group multiplication and hence a subgroup.
\end{proof}

The proof of the next proposition is along the lines of Prop. \ref{coset_stab}.\\

\begin{proposition}
    \label{zero_ent_suff}
Let $G$ be a finite group, and $\mu$ be a shift-invariant Borel probability measure on $G^{\mathbb{N}_0}$ invariant also relative to $\tau_G$. Assume that there exists $H\vartriangleleft G$ for which $H_{\mu,x}=H$ ($H_{\mu,x}$ as in Theorem \ref{group_th}) $\mu$-almost-surely.
If $C$ is the maximal size of a proper subgroup of $H$ and
$\frac{\left|G\right|}{\left|H\right|}<\frac{\left|H\right|}{C}$
then $h_{\left( \pi_{\mu} \right)_{*}\mu}\left(\sigma\right)=0$.
\end{proposition}

\begin{proof}

Notice that $U_{\mu,x}$ of Prop. \ref{coset_stab} (ii) is $\mu$-almost-surely equal
to $H$ since there are less than $\frac{\left|H\right|}{\left|U_{\mu,x}\right|}$ different cosets of $H$ in $G$ for any $U_{\mu,x}<H$. So $\mu$-almost-surely the coset of $H$ at index $0$ of $x$ given its values at indices $1,2,3,4,\dots$ does not depend on the value chosen at index $1$ from the coset of $H$ that the values at indices $2,3,4,\dots$ determine at index $1$. We exploit this idea to prove by induction on $1\leq k\in{\mathbb{N}_0}$ that $\mu$-almost-surely the values of $x$ at indices $k,k+1,k+2,\dots$ determine a specific coset of $H$ at each of the indices $0,1,2,\dots,k-1$.\\

$k=1$ is just Theorem \ref{group_th}. For $k>1$, assume the
induction hypothesis holds for $k-1$. Thus $\mu$-almost-surely specific
cosets $Hx_{1},Hx_{2},\dots,Hx_{k-1}$ at indices $1,2,\dots,k-1$
are determined by the values at indices $k,k+1,k+2,\dots$ of $x$.
As for the index $0$, by an argument similar to the proof of Prop.
\ref{coset_stab}, it follows that there exists a subgroup $U'_{\mu,x}\leq H$ that
its right cosets in $G$ that are contained in $Hx_{k-1}$ form exactly
its partition into the possibilities of cosets of $H$ they permit
to appear at index $0$ (the cosets at indices $1,\dots, k-1$ are already determined by the induction hypothesis). Then $U'_{\mu,x}$ must be equal to $H$,
i.e. also the coset of $H$ at index $0$ is determined by the values at indices $k,k+1,k+2,\dots$ of $x$.\\

So in particular $x\mapsto\left(\pi_\mu\left(x\right)\right)_0$ is $\sigma^{-k}\mathcal{B}$-measurable for all $k\in {\mathbb{N}_0}$. Applying Lemma \ref{zero_ent} to the statement we have just proved establishes
the claim.
\end{proof}

Let us apply Prop. \ref{zero_ent_suff} to a examples.\\

\begin{example}
     For a positive odd integer $m$ we take $G = \mathbb{Z}/2m\mathbb{Z}$. Given $\mu$ as in Theorem \ref{group_th}, if $H_{\mu,x}$ is $\mu$-almost-surely equal to $G = 2\mathbb{Z}/2m\mathbb{Z}$ then the
conditions of Prop. \ref{zero_ent_suff} are satisfied and $h_{\left( \pi_{\mu} \right)_{*}\mu}\left(\sigma\right)=0$.
$\pi_\mu$ is just the component-wise canonical projection $\mathbb{Z}/2m \rightarrow  \left(\mathbb{Z}/2m\mathbb{Z}\right)/\left(2\mathbb{Z}/2m\mathbb{Z}\right) \cong \mathbb{Z}/2\mathbb{Z}$. The measure corresponding to $\left( \pi_{\mu} \right)_{*}\mu$ on $\left(\mathbb{Z}/2\mathbb{Z}\right)^{{\mathbb{N}_0}}$ is then also invariant relative to the group cellular automaton of $\mathbb{Z}/2\mathbb{Z}$. We thus obtain for the Ledrappier cellular automaton a shift-invariant measure which is also invariant relative to that automaton and possesses zero entropy relative to each of those two maps. We do not know any such measure that is atom-less.
\end{example}

\begin{example}
\label{dihedral}
    All is similar to the previous example if $\mathbb{Z}/m\mathbb{Z}$ is replaced by the Dihedral group $D_m$ (for odd $m$) and the subgroup $2\mathbb{Z}/m\mathbb{Z}$ is replaced by the subgroup of rotations (which is normal).
\end{example}

\vspace{0.3cm}

The next section is devoted to a more general analysis of the case $h_{\left( \pi_{\mu} \right)_{*}\mu}\left(\sigma\right)=0$ for bi-permutative cellular automata. Theorem \ref{zero_ent_synth} and Remark \ref{converse_rem_to_synthesis_thrm}, Cor. \ref{group_cor} and Remark \ref{group_conv_rem}  relate to the converse of the construction of the zero entropy measures (relative to each of the two maps) in the two Examples.

\vspace{0.3cm}

\section{Invariant Mearures of Bi-Permutative Cellular Automata with a Structural Factor of Zero Entropy}

\label{zero_ent_section}

The special case in the setting of Theorem \ref{biper_th} in which $h_{\left( \pi_{\mu} \right)_{*}\mu}\left(\sigma\right)=0$
is susceptible to further analysis.\\

First, we aim to show that in this case, not only the shift, but also the bi-permutative cellular automaton $\tau$ induces a $\left(\pi_{\mu}\right)_*\mu$-invariant map on the probability space $\left(\left(2^\Lambda\right)^{\mathbb{N}_0},\left( \pi_{\mu} \right)_{*}\mu\right)$.\\

We define a few maps.\\

\begin{definition}
We denote by 

$\psi_{1}:\left\{ \left(A,a\right)\,:\,A\subseteq\Lambda,\,a\in A\right\} ^{\mathbb{N}_0}\rightarrow\left(2^\Lambda\right)^{\mathbb{N}_0}$
the component-wise projection $\left(A,a\right)\mapsto A$,

and by
$\psi_{2}:\left\{ \left(A,a\right)\,:\,A\subseteq\Lambda,\,a\in A\right\} ^{\mathbb{N}_0}\rightarrow \Lambda^{\mathbb{N}_0}$
the component-wise projection $\left(A,a\right)\mapsto a$.    
\end{definition}

\vspace{0.3cm}

\begin{definition}
Given the conditions of Theorem \ref{biper_th}, we denote by $\varphi_\mu$ the
map from $\Lambda^{\mathbb{N}_0}$
to $\left\{ \left(A,a\right)\,:\,A\subseteq\Lambda,\,a\in A\right\} ^{\mathbb{N}_0}$
defined $\mu$-almost-surely as $\left(\varphi_\mu\left(x\right)\right)_{i}=\left(\left(\pi_{\mu}\left(x\right)\right)_{i},x_i\right)$.    
\end{definition}

\vspace{0.3cm}

In particular,  $\psi_1\circ\varphi_\mu = \pi_{\mu}$. Also, equipping the space $\left\{ \left(A,a\right)\,:\,A\subseteq\Lambda,\,a\in A\right\} ^{\mathbb{N}_0}$
with the measure $\left(\varphi_\mu\right)_{*}\mu$, notice that $\psi_{2}$
and $\varphi_\mu$ are measure-theoretically inverses of each other.\\

\begin{definition}
Given any cellular automaton $\tau:\Lambda^{\mathbb{N}_0}\rightarrow\Lambda^{\mathbb{N}_0}$, we define a new cellular automaton $\tau':\left(2^\Lambda\right)^{\mathbb{N}_0}\rightarrow\left(2^\Lambda\right)^{\mathbb{N}_0}$ by \[ \left(\tau'\left(\left(A_0,A_1,A_2,\dots\right)\right)\right) _i = \{\left(\tau\left(\left(a_0,a_1,a_2,\dots\right)\right)\right)_i\,:\,\forall i \,\, a_i\in A_i\}.\]     
\end{definition}

\vspace{0.3cm}

\begin{theorem}
\label{zero_ent_analysis}
Let $\tau:\Lambda^{\mathbb{N}_0}\rightarrow\Lambda^{\mathbb{N}_0}$ be a bi-permutative cellular automaton, and let $\mu$ be a shift-invariant Borel probability measure on $\Lambda^{\mathbb{N}_0}$ which is also $\tau$-invariant (this is the setting of Theorem \ref{biper_th}). Assume also $h_{\left( \pi_{\mu} \right)_{*}\mu}\left(\sigma\right)=0$. We define a shift-invariant measure $\tilde{\mu}$ on $\left\{ \left(A,a\right)\,:\,A\subseteq\Lambda,\,a\in A\right\} ^{{\mathbb{N}_0}}$ by requiring that for a cylinder $\tilde{\mu}\left(\left[\left(A_{0},a_{0}\right),\dots,\left(A_{n},a_{n}\right)\right]\right)=\frac{\left( \pi_{\mu} \right)_{*}\mu\left(\left[A_{0},\dots,A_{n}\right]\right)}{\left|A_{0}\right|\cdots\left|A_{n}\right|}$. Then:\\ 
\textbf{(i):} $\varphi_\mu$ is a measure-theoretical isomorphism between the dynamical systems $\left(\Lambda^{\mathbb{N}_0},\sigma,\mu\right)$
and $\left(\left\{ \left(A,a\right)\,:\,A\subseteq\Lambda,\,a\in A\right\} ^{{\mathbb{N}_0}},\sigma,\tilde{\mu}\right)$, i.e. $\left(\varphi_\mu\right)_*\mu = \tilde{\mu}$.\\
\textbf{(ii):} $\left(\left(2^{\Lambda}\right)^{{\mathbb{N}_0}},\sigma,\left( \pi_{\mu} \right)_{*}\mu\right)$ is the Pinsker factor of $\left(\Lambda^{\mathbb{N}_0},\sigma,\mu\right)$.\\
\textbf{(iii):} Measure-theoretically $\pi_{\mu}\circ\tau = \tau'\circ\pi_{\mu}$ (in particular, the measure $\left(\pi_{\mu}\right)_*\mu$ on $\left(2^{\Lambda}\right)^{{\mathbb{N}_0}}$ is also invariant relative to $\tau'$).\\
\textbf{(iv):} If  $\mu$  is ergodic relative to the joint action of $\sigma$ and $\tau$ then there exists an integer $1\leq k\leq \left|\Lambda\right|$ so that  all components of a point in $\left(2^{\Lambda}\right)^{\mathbb{N}_0}$ are $\left( \pi_{\mu} \right)_{*}\mu$-almost-surely subsets of size $k$.
\end{theorem}

\begin{proof}

\textbf{(i):} We know that
\[
\left(\varphi_\mu\right)_*\mu\left(\left[\left(A_{0},a_{0}\right),\dots,\left(A_{n},a_{n}\right)\right]\right)
\]
\[
=\left(\varphi_\mu\right)_*\mu\left(\left[\left(A_{0},\cdot\right),\dots,\left(A_{n},\cdot\right)\right]\right)\cdot\left(\varphi_\mu\right)_*\mu\left(\left[\left(A_{0},a_{0}\right),\dots,\left(A_{n-1},a_{n}\right)\right]\,|\,\left[\left(A_{0},\cdot\right),\dots,\left(A_{n},\cdot\right)\right]\right)
\]

\[
=\frac{\left( \pi_{\mu} \right)_{*}\mu\left(\left[A_{0},\dots,A_{n}\right]\right)}{\left|A_{0}\right|\cdots\left|A_{n}\right|},
\]

where the equality \[\left(\varphi_\mu\right)_*\mu\left(\left[\left(A_{0},a_{0}\right),\dots,\left(A_{n},a_{n}\right)\right]\,|\,\left[\left(A_{0},\cdot\right),\dots,\left(A_{n},\cdot\right)\right]\right)=\frac{1}{\left|A_{0}\right|\cdots\left|A_{n}\right|}\]
is because $h_{\left( \pi_{\mu} \right)_{*}\mu}\left(\sigma\right)=0$ (given a point in $\Lambda^{\mathbb{N}_0}$ for which the first values of its image by $\pi_\mu$ are $A_0,\dots,A_n$,  modifying $a_n$ inside $A_n$ cannot alter $A_{n-1}$ etc.). Hence $\left(\varphi_\mu\right)_*\mu=\tilde{\mu}$.\\

\textbf{(ii):} Let $E$ be a tail event with respect to the partition according to the value at index zero of the system $\left(\Lambda^{\mathbb{N}_0},\sigma,\mu\right)$. We need to prove that $E\in\pi_{\mu}^{-1}\mathcal{B}_{\left(2^{\Lambda}\right)^{{\mathbb{N}_0}}}$. This is equivalent to proving that the conditional expectation $\mathbb{E}_\mu \left( \chi_E\,|\,\pi_{\mu}^{-1} \mathcal{B}_{\left(2^{\Lambda}\right)^{{\mathbb{N}_0}}} \right) \left(x\right)$ is $\mu$-almost-surely equal to either $0$ or $1$. If $\varphi_\mu\left(x\right) = \left( \left( A_0, x_0\right), \left(A_1, x_1 \right), \left( A_2,x_2 \right),\dots \right)$
 then the atom of $x$ in the countably generated $\sigma$-algebra $\pi_{\mu}^{-1}\mathcal{B}_{\left(2^{\Lambda}\right)^{{\mathbb{N}_0}}}$ corresponds, under the identification made by $\varphi_\mu$, to $ \left( \left( A_0, \cdot\right), \left(A_1, \cdot \right), \left( A_2,\cdot \right),\dots \right)$. The conditional probability $\mu_x^{\pi_{\mu}^{-1}\mathcal{B}_{\left(2^{\Lambda}\right)^{{\mathbb{N}_0}}} }$ on such an atom is $\mu$-almost-surely identifiable with the probability measure on $A_0\times A_1\times A_2 \times \dots$ that represents independent uniformly random choices on $A_0,A_1,A_2,\dots$. Just like on a Bernoulli system, any tail-event of this non-stationary system has to be of probability either $0$ or $1$. When this is the case, as the interesection of $E$ with that atom is a tail event in $A_0\times A_1\times A_2\times \dots$, the value of $\mu_x^{\pi_{\mu}^{-1}\mathcal{B}_{\left(2^{\Lambda}\right)^{{\mathbb{N}_0}}}} \left( E \right)$  is either $0$ or $1$.\\

\textbf{(iii):} To prove that the fibers of $\pi_{\mu}$ are measure-theoretically
respected by $\tau$, it suffices to show that $\left( \pi_{\mu} \right)_{*}\mu$-almost-surely $\left(\left(\pi_{\mu}\left(x\right)\right)_{0},\left(\pi_{\mu}\left(x\right)\right)_{1},\left(\pi_{\mu}\left(x\right)\right)_{2},\dots\right)$
 determines $\left(\pi_{\mu}\left(\tau\left(x\right)\right)\right)_{0}$. We know that $\left( x_0,x_1,x_2,\dots \right)$ $\mu$-almost-surely does, and that it does so in a manner that does not depend on any of the first $x_{0},\dots,x_{n}$ because $\left(\left(\pi_{\mu}\left(\tau\left(x\right)\right)\right)_{n+1},\left(\pi_{\mu}\left(\tau\left(x\right)\right)\right)_{n+2},\left(\pi_{\mu}\left(\tau\left(x\right)\right)\right)_{n+3},\dots\right)$ are still determined and $h_{\left( \pi_{\mu} \right)_{*}\mu}\left(\sigma\right)=0$. Thus it just depends on the Pinsker $\sigma$-algebra of $\left(\Lambda^{\mathbb{N}_0},\sigma,\mu\right)$ and by (ii) we are done.\\

This means that there exists a map $\eta: \left(2^{\Lambda}\right)^{{\mathbb{N}_0}}\rightarrow\left(2^{\Lambda}\right)^{{\mathbb{N}_0}}$ defined $\left(\pi_{\mu}\right)_*\mu$-almost-surely so that measure-theoretically $\pi_{\mu}\circ\tau = \eta\circ\pi_{\mu}$. But then $\left(\eta\left(z\right)\right)_0$ cannot be anything more or less than $\left(\tau'\left(z\right)\right)_0$ for $z\in \left(2^{\Lambda}\right)^{{\mathbb{N}_0}}$ $\left(\pi_{\mu}\right)_*\mu$-almost-surely.\\

\textbf{(iv):} The function $x \mapsto \left| \Lambda_{\mu,x} \right|$ is $\tau$-invariant (using the notation of Theorem \ref{biper_th}) and hence, by (iii) and by the right permutativity of $\tau$, a point $x$ with $\pi \left(x\right) = \left(A_0,A_1,A_2,\dots\right) \in \left(2^\Lambda\right)^{\mathbb{N}_0}$  satisfies  $\mu$-almost-surely that $\left|A_i\right| \geq  \left|A_{i+1}\right|$ for all $i$. This must stabilize on some size $k_x$ and $k_x$ is invariant relative to both $\sigma$ and $\tau$, thus ergodicity assures that there exists a $k$ so that $\mu$-almost-surely $k_x = k$ .
\end{proof}

\begin{definition}
    Given a cellular automaton $\tau:\Lambda^{\mathbb{N}_0}\rightarrow\Lambda^{\mathbb{N}_0}$ and an integer $1\leq k\leq \Lambda$, let us denote by $Z_{\tau,k}$ the subshift of $\left(2^{\Lambda}\right)^{{\mathbb{N}_0}}$ composed of the points $\left( A_0,A_1,A_2,\dots \right)\in \left(2^{\Lambda}\right)^{{\mathbb{N}_0}}$ for which $\left|A_0\right|=\left|A_1\right|=\left|A_2\right|=\dots=k$ and  $\left|\left(\left(\tau'\right)^n \left( A_0,A_1,A_2,\dots \right)\right)_i\right| = k$ for all $n>0$ and indices $i$. Also, let us denote $\cup_{k=1}^{\left| \Lambda \right|} Z_{\tau,k}$ by $Z_\tau$.\\
    
\end{definition}

\vspace{0.3cm}

In the setting of Theorem \ref{zero_ent_analysis}, by its part (iv), if $\mu$ is ergodic relative to the joint action of $\sigma$ and $\tau$, the support of $\left(\pi_{\mu}\right)_*\mu$  is contained in some $Z_{\tau,k}$. If $\mu$ is not ergodic relative to that joint action, we can still apply ergodic decomposition to deduce that the support of $\left(\pi_{\mu}\right)_*\mu$ is contained in $Z_\tau$.\\

$Z_\tau$ is never empty as it always contains the subshift composed of the points that all of their components are singletons and also the point $\left(\Lambda,\Lambda,\Lambda,\dots\right)$. In the setting of Theorem \ref{zero_ent_analysis}, if $\left(\pi_{\mu}\right)_*\mu$ is $\delta_{\left(\Lambda,\Lambda,\Lambda,\dots\right)}$ then $\mu$ is the uniform measure on $\Lambda^{\mathbb{N}_0}$, and if $\left(\pi_{\mu}\right)_*\mu$ is supported on the set of points that are composed only of singletons then $\mu$ is of zero entropy relative to the shift (equivalently, to $\tau$). We now show a relevant example.\\

\textbf{Example:} Let us take $\Lambda=\left\{ A,B,C\right\} $ with
$r$ being the commutative binary operation defined by the following
diagram (the labels on the edges are the outcomes). \\

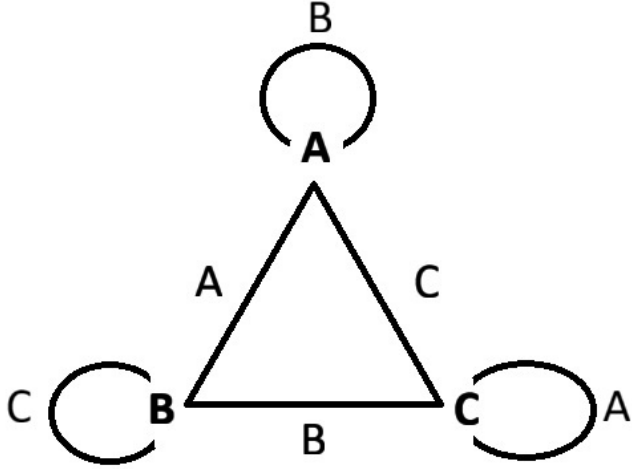
\begin{center}
\begin{tikzpicture}[
    scale=1.5,
    line width=1.1pt,
    every node/.style={font=\large}
]

\coordinate (A) at (0,2);
\coordinate (B) at (-1.2,0);
\coordinate (C) at (1.2,0);

\draw (A) -- (B)
      node[midway,left=6pt] {$A$};

\draw (B) -- (C)
      node[midway,below=6pt] {$B$};

\draw (C) -- (A)
      node[midway,right=6pt] {$C$};

\node[above=3pt] at (A) {$\mathbf{A}$};
\node[left=3pt]  at (B) {$\mathbf{B}$};
\node[right=3pt] at (C) {$\mathbf{C}$};

\draw
    (-0.22,2.38)
    .. controls (-0.85,2.70) and (-0.65,3.35) ..
    (0,3.35)
    .. controls (0.65,3.35) and (0.85,2.70) ..
    (0.22,2.38);

\node at (0,3.65) {$B$};

\draw
    (-1.55,0.25)
    .. controls (-2.20,0.55) and (-2.25,-0.35) ..
    (-1.75,-0.55)
    .. controls (-1.35,-0.72) and (-1.05,-0.45) ..
    (-1.10,-0.18);

\node at (-2.28,-0.02) {$C$};

\draw
    (1.55,0.25)
    .. controls (2.20,0.55) and (2.25,-0.35) ..
    (1.75,-0.55)
    .. controls (1.35,-0.72) and (1.05,-0.45) ..
    (1.10,-0.18);

\node at (2.25,-0.02) {$A$};

\end{tikzpicture}
\end{center}

This operation is non-associative, so it is not a multiplication operation of a group, but it does form a rule (as in Def. \ref{biper_def}) for a bi-permutative cellular automaton $\tau:\{A,B,C\}^{\mathbb{N}_0}\rightarrow\{A,B,C\}^{\mathbb{N}_0}$. Let $\mu$ be a measure as in Theorem \ref{zero_ent_analysis} for this $\tau$, which is ergodic under the joint action of $\sigma$ and $\tau$. We claim that the support of $\left(\pi_*\right)\mu$ is contained either in the set of points with only singletons in their components or is just the point $\left(\left\{ A,B,C\right\} ,\left\{ A,B,C\right\} ,\left\{ A,B,C\right\} ,\dots\right)$. Assume it is not, then it is contained in $Z_{\tau,2}$, but this cannot be since $Z_{\tau,2}$:

\[
\left\{ r\left(x,y\right)\,:\,x\in\left\{ A,B\right\} ,y\in\left\{ A,B\right\} \right\} =\left\{ A,B,C\right\} ,
\]
\[
\left\{ r\left(x,y\right)\,:\,x\in\left\{ A,C\right\} ,y\in\left\{ A,C\right\} \right\} =\left\{ A,B,C\right\} ,
\]
\[
\left\{ r\left(x,y\right)\,:\,x\in\left\{ B,C\right\} ,y\in\left\{ B,C\right\} \right\} =\left\{ A,B,C\right\} ,
\]
\[
\left\{ r\left(x,y\right)\,:\,x\in\left\{ A,B\right\} ,y\in\left\{ A,C\right\} \right\} =\left\{ A,B,C\right\} ,
\]
\[
\left\{ r\left(x,y\right)\,:\,x\in\left\{ A,B\right\} ,y\in\left\{ B,C\right\} \right\} =\left\{ A,B,C\right\} ,
\]
\[
\left\{ r\left(x,y\right)\,:\,x\in\left\{ A,C\right\} ,y\in\left\{ B,C\right\} \right\} =\left\{ A,B,C\right\} .
\]
\\\\

The next Proposition is a converse of the previous one (from analysis to synthesis).\\

\begin{theorem}
\label{zero_ent_synth}

Let $\tau:\Lambda^{\mathbb{N}_0}\rightarrow\Lambda^{\mathbb{N}_0}$ be a left permutative cellular automaton and let $\nu$ be a shift-invariant Borel probability measure on $Z_\tau$ that is also invariant relative to $\tau'$. Form a measure $\tilde{\mu}$ on $\left\{ \left(A,a\right)\,:\,A\subseteq\Lambda,\,a\in A\right\} ^{\mathbb{N}_0}$ by the requirement $\tilde{\mu}\left(\left[\left(A_{0},a_{0}\right),\dots,\left(A_{n},a_{n}\right)\right]\right) = \frac{\nu\left(\left[A_{0},\dots,A_{n}\right]\right)}{\left|A_{0}\right|\cdots\left|A_{n}\right|}$ and define $\mu = \left(\psi_2\right)_*\tilde{\mu}$. Then:\\
\textbf{(i):} $\mu$ is invariant relative to both $\sigma$ and $\tau$.\\
\textbf{(ii):} If $\tau$ is bi-permutative and $h_\nu\left( \sigma \right) = 0$ then $\left(\pi_{\mu}\right)_*\mu = \nu$. \footnote{The proof of this claim written bellow is similar to the argument that lies behind "the easily checked fact" phrase opening the proof of Theorem 4.2 in \cite{key-10}.}\\
\textbf{(iii):} If in addition to the premises of (ii),  $\nu$ is also ergodic relative to the joint action of the shift and $\tau'$ then $\mu$ is ergodic relative to the joint action of the shift and $\tau$.
\end{theorem}

\begin{proof}
\textbf{(i):} We define a map $\tilde{\tau}:\left\{ \left(A,a\right)\,:\,A\subseteq\Lambda,\,a\in A\right\} ^{\mathbb{N}_0}\rightarrow \left\{ \left(A,a\right)\,:\,A\subseteq\Lambda,\,a\in A\right\} ^{\mathbb{N}_0}$ as \[\left(\tilde{\tau}\left(  \left(A_0,a_0\right),\left(A_1,a_1\right),\left(A_2,a_2\right),\dots \right) \right)_i = \left(\left( \tau'\left(A_0,A_1,A_2,\dots \right) \right)_i, \left( \tau\left(a_0,a_1,a_2,\dots \right) \right)_i\right).\]

It suffices to prove that $\tilde{\mu}$ is invariant relative to both the shift and $\tilde{\tau}$. The invariance relative to the shift is immediate from the invariance of $\nu$ relative to its shift. To see the invariance of $\tilde{\mu}$  relative to $\tilde{\tau}$, suppose we are given a cylinder $\left[\left(A_0,a_0\right),\left(A_1,a_1\right),\dots \left(A_n,a_n\right)\right]$. Its $\tilde{\mu}$-probability is 
\[  \frac{\nu\left(\left[A_{0},\dots,A_{n}\right]\right)}{\left|A_{0}\right|\cdots\left|A_{n}\right|}. \] \\ 
By the invariance of $\nu$  relative to $\tau'$ this is equal to 
\[   \frac{\nu\left(\left(\tau'\right)^{-1}\left(\left[A_{0},\dots,A_{n}\right]\right)\right)}{\left|A_{0}\right|\cdots\left|A_{n}\right|} =  \sum_{  \{\left(B_0,\dots,B_{n+1}\right)\  :\, \left[B_0,\dots,B_{n+1}\right] \subseteq \left(\tau'\right)^{-1}\left(\left[A_{0},\dots,A_{n}\right]\right)\}} \frac{\nu \left( \left[B_0,\dots,B_{n+1}\right] \right)}{\left|A_{0}\right|\cdots\left|A_{n}\right|}  \] \\

\[ =   \sum_{  \{\left(B_0,\dots,B_{n+1}\right)\  :\, \left[B_0,\dots,B_{n+1}\right] \subseteq \left(\tau'\right)^{-1}\left(\left[A_{0},\dots,A_{n}\right]\right)\}} \frac{\nu \left( \left[B_0,\dots,B_{n+1}\right] \right)}{\left|B_{0}\right|\cdots\left|B_{n}\right|}  \]

\[ =   \sum_{ \{\left(B_0,\dots,B_{n+1}\right)\  :\, \left[B_0,\dots,B_{n+1}\right] \subseteq \left(\tau'\right)^{-1}\left(\left[A_{0},\dots,A_{n}\right]\right)\}} \frac{\left|B_{n+1}\right| \nu \left( \left[B_0,\dots,B_{n+1}\right] \right)}{\left|B_{0}\right|\cdots\left|B_{n+1}\right|}, \]

where we made use of the fact that $\left|A_0\right|=\left|B_0\right|,\dots,\left|A_n\right|=\left|B_n\right|$ that holds for any non-zero summand because $\nu$ is supported on $Z_\tau$.\\

By left permutativity, the number of elements in each $B_0\times\dots\times B_{n+1}$ that the automaton rule sends to $\left(a_0,\dots,a_n\right)$ is equal to $\left|B_{n+1}\right|$ (there is one such for every choice of element of $B_{n+1}$). Thus the last expression is exactly the $\tilde{\mu}$-probability of the pre-image of our cylinder under $\tilde{\tau}$.\\

\textbf{(ii):} The support of $\mu_{\psi_2\left(y\right)} ^{\sigma^{-1}\mathcal{B}}$, as a measure on $\Lambda$, contains $\tilde{\mu}$-almost-surely the set $\left(\psi_1\left(y\right)\right)_0$, i.e. $\left(\psi_1\left(y\right)\right)_0 \subseteq \left(\pi_\mu \left(\psi_2\left(y\right)\right)\right)_0 $ $\tilde{\mu}$-almost-surely. It is left to prove that these two sets are in fact $\tilde{\mu}$-almost-surely equal.\\

Now, $h_{\mu} \left(\sigma\right) = \mathbb{E}_{\tilde{\mu}} \log \left(\left|\left(\pi_\mu \left(\psi_2\left(y\right)\right)\right)_0\right|\right)$, and
$h_\nu \left(\sigma\right) = 0$ implies that also $ h_{\tilde{\mu}} \left(\sigma\right) = \mathbb{E}_{\tilde{\mu}} \log \left(\left|\left(\psi_1\left(y\right)\right)_0\right|\right)$. Combining these equalities with the inequality $h_{\mu} \left(\sigma\right) \leq h_{\tilde{\mu}} \left(\sigma\right)$ yields
\[\mathbb{E}_{\tilde{\mu}} \left(\left|\left(\pi_\mu \left(\psi_2\left(y\right)\right)\right)_0\right|\right) \leq \mathbb{E}_{\tilde{\mu}} \left(\left|\left(\psi_1\left(y\right)\right)_0\right|\right).\]

Since $\left|\left(\pi_\mu \left(\psi_2\left(y\right)\right)\right)_0\right|\geq \left|\left(\psi_1\left(y\right)\right)_0\right|$ holds $\tilde{\mu}$-almost-surely, we deduce that  $\left|\left(\pi_\mu \left(\psi_2\left(y\right)\right)\right)_0\right| = \left|\left(\psi_1\left(y\right)\right)_0\right|$ holds $\tilde{\mu}$-almost-surely. Thus we obtain the equality of the two sets stated above.\\

\textbf{(iii):} Applying an ergodic decomposition to $\mu$ with respect to the joint action of the shift and $\tau$ - and considering Theorem \ref{zero_ent_analysis} (i) and the assumed ergodicity of $\nu$ - the ergodic components must $\mu$-almost-surely be equal to $\mu$.
\end{proof}

\begin{remark}
\label{converse_rem_to_synthesis_thrm}
We do not know whether given $\nu$ as
in Theorem \ref{zero_ent_synth} (iii) that is already ergodic relative to $\tau'$ guarantees that $\mu$ is already ergodic relative to $\tau$. 
\end{remark}

The next corollary more or less sums up the last two theorems.\\

\begin{corollary}
\label{1-1_cor}

Let $\tau:\Lambda^{\mathbb{N}_0}\rightarrow\Lambda^{\mathbb{N}_0}$ be a bi-permutative cellular automaton. Then there exists a bijective correspondence between the collection of shift-invariant Borel probability measures $\mu$ on $\Lambda^{\mathbb{N}_0}$ that are also invariant relative to $\tau$, ergodic relative to the joint action of the shift and $\tau$ and satisfy $h_{\left(\pi_{\mu}\right)_*\mu}\left(\sigma\right) = 0$, and the collection of shift-invariant Borel probability measures $\nu$ on $Z_\tau$ that are also invariant relative to $\tau'$, ergodic relative to the joint action of the shift and $\tau'$ and satisfy $h_{\nu}\left(\sigma\right) = 0$.

\end{corollary}

\subsection{Invariant Measures of Group Multiplication Cellular Automata with a Structural Factor of Zero Entropy}

We now consider the special case of a group multiplication cellular
automaton $\tau_G$ for a finite group $G$. Here we can add on the content of Theorem \ref{zero_ent_analysis}.\\

\begin{theorem}
\label{zero_ent_group_thrm}

Let $G$ be a finite group. If $\mu$ is a shift-invariant Borel probability measure which is
also invariant relative to $\tau_G$, ergodic relative to the joint action and satisfies $h_{\left( \pi_{\mu} \right)_{*}\mu}\left(\sigma\right)=0$, then there exists a subgroup $H_\mu \leq G$ so that the support of $\left( \pi_{\mu} \right)_{*}\mu$ is contained
in $\left\{ H_\mu g\,:\,gH_\mu g^{-1}=H_\mu \right\} ^{{\mathbb{N}_0}}$.\\

\end{theorem}

\begin{remark}
    This implies that measure-theoretically $\tau_{N_G \left(H_\mu \right) / H_\mu}\circ \pi_\mu = \pi_\mu \circ \tau_G$, where $N_G \left(H_\mu \right)$ is the normalizer of $H_\mu$ in $G$.\\
\end{remark}

\begin{proof}
The support of $\left(\pi_\mu\right)_*\mu$ is contained in $Z_{\tau_{G},k}$ for some integer $1\leq k\leq \left| \Lambda \right|$. So a typical point has in all of its components only right cosets of size $k$.\\

Suppose we are given such a typical point. There necessarily exists  a subroup $H_\mu \leq G$ whose right cosets appear in an infinite set of indices in this point. Assume the point has in three consecutive indices cosets $H_\mu g_{1},Ug_{2},Kg_{3}$
of three subgroups $H_\mu ,U,K\leq G$ of size $k$.\\

$H_\mu g_{1}Ug_{2}$ must also be a right coset of $H_\mu $ (as $H_{\mu,x}$ is $\tau_G$-invariant). Thus $H_\mu g_{1}Ug_{2}=H_\mu g_{1}g_{2}$ which in turn is equivalent to $H_\mu g_{1}Ug_{1}^{-1}=H_\mu $. This is true if and only if $g_{1}Ug_{1}^{-1}\subseteq H_\mu $, and the latter condition is equivalent to $g_{1}Ug_{1}^{-1} = H$ by the assumption $\left|H_\mu\right|=\left|U\right|$.\\

Applying the multiplication automaton element-wise we obtain $H_\mu g_{1}Ug_{2}=H_\mu g_{1}g_{2}$
in first of the components and $Ug_{2}g_{3}$ in the next one. Thus, by repeating the argument from the last paragraph, we obtain $g_{2}^{-1}g_{1}^{-1}Ug_{1}g_{2}=H_\mu $ which implies $g_{2}^{-1}H_\mu g_{2}=H_\mu $. Thus it can be shown inductively that $gH_\mu g^{-1}=H_\mu $ for all elements $g$ belonging to any coset that appears as a component somewhere after those $H_\mu g_{1},Ug_{2}$ in the typical point we started with. Take a coset of $H_\mu$ appearing after those $H_\mu g_{1},Ug_{2}$ we started with whose left-most neighbor is a coset  $U'g'$ of some subgroup $U'$ of the same size. Then on the one hand $g'H_\mu \left(g'\right)^{-1}=H_\mu $
but on the other hand $g'H_\mu \left(g'\right)^{-1}=U'$. Thus $H=U'$.\\

We have proved that from some index and onward in our typical point appear only cosets of $H_\mu $ and also that they all belong to $\left\{ H_\mu g\,:\,g\in G,\,gH_\mu g^{-1}=H_\mu \right\} ^{\mathbb{Z}}$.
\end{proof}

So for group multiplication cellular automata, we arrived at a strengthened version of Corollary \ref{1-1_cor}.

\begin{corollary}
\label{group_cor}
Let $G$ be a finite group. Then there exists a bijective correspondence between the collection of shift-invariant Borel probability measures $\mu$ on $G^{\mathbb{N}_0}$  that are also invariant relative to $\tau_G$, ergodic relative to the joint action of the shift and $\tau_G$ and satisfy $h_{\left(\pi_{\mu}\right)_*\mu}\left(\sigma\right) = 0$, and the collection of shift-invariant Borel probability measures $\nu$ on any of the spaces $\left\{ Hg\,:\,gHg^{-1}=H\right\} ^{{\mathbb{N}_0}}$ - for subgroups $H\leq G$ -  that are also invariant relative to $\tau_G'$, ergodic relative to the joint action of the shift and $\tau_G'$ and satisfy $h_{\nu}\left(\sigma\right) = 0$.

\end{corollary}

\begin{remark}
\label{group_conv_rem}
Recall Example \ref{kitchens}. If we are given a shift-invariant Borel probability
measure on $\left(\mathbb{Z}/2\mathbb{Z}\right)^{{\mathbb{N}_0}}$ which is also invariant relative to the group multiplication cellular automaton (the Ledrappier cellular automaton), ergodic relative to the joint action of the two maps, and possesses zero entropy relative to each of them, then we can construct a such a measure on $\left(D_{p}\right)^{{\mathbb{N}_0}}$ except that it has positive entropy relative to each of the two maps. This latter measure is always atom-less. However notice that by Remark \ref{converse_rem_to_synthesis_thrm} we do not know whether ergodicity already relative to just the Ledrappier cellular automaton has to imply ergodicity of the measure on $\left(D_{p}\right)^{{\mathbb{N}_0}}$ already relative to just the group cellular automaton.

\end{remark}

Thus we have reduced the problem to the zero entropy case of a group multiplication cellular automaton. However, it is yet unknown even whether such atom-less measures exist.\\

\section{Generalizing the Results Beyond Bi-Permutative Cellular Automata}
\label{sect_RLP}

We now generalize our results to a larger class of cellular automata and two-dimensional subshifts that we shall refer to as \textit{RLP subshifts}.\\

The situation is rather peculiar in that we have only one type of example of a system that is proven to belong to the RLP class but is not a bi-permutative cellular automata. However, this class might contain many more systems.\\

\subsection{RLP Subshifts}

We begin with a few definitions that will lead our way to defining RLP subshifts of the two-dimensional shift system $\Lambda^{{\mathbb{N}_0}^2}$ (RLP stands for reciprocal and left permutative).\\

The definitions \ref{generation_def} and \ref{LP_def} are formulated for the vector $\left(1,0\right)$, but we shall assume similar definitions also for $\left(0,1\right)$.\\

\underline{Convention:} For simplicity, let us agree that when we refer to any two-dimensional subshift $W$ (of some alphabet) with index set ${\mathbb{N}_0}^2$, we shall always implicitly assume that the maps $\sigma_{\left(1,0\right)}:W\rightarrow W$ and $\sigma_{\left(0,1\right)}:W\rightarrow W$ are onto.\\

\begin{definition}
\label{generation_def}

    Given a two-dimensional subshift $W\subseteq \Lambda^{{\mathbb{N}_0}^2}$ we shall say that $\left(1,0\right)$ \textit{generates} $W$ if the values of a point at the indices $\{\left(m,0\right)\}_{m\in {\mathbb{N}_0}}$ determine the values at the rest of the indices. We shall say that $\left(1,0\right)$ \textit{almost-generates} $W$ if $W$ is uncountable and the just mentioned property is true outside of some countable set of points. And if an invariant measure on $W$ is involved, we shall say that $\left(1,0\right)$ \textit{measure-theoretically-generates} $W$ if this is true outside of a set of zero measure.\\
\end{definition}

  If $W$ is almost-generated by $\left(1,0\right)$ then such a countable set can be taken to be the set of all points with a zero row shared with some other point. Let us denote this set by $C_W$. \\

The next proposition adds clarity to the various generation definitions above.\\

\begin{proposition}
\label{generation_continuity}

Let $W\subseteq \Lambda^{{\mathbb{N}_0}^2}$ be a two-dimensional subshift. And denote by $W_0 \subseteq \Lambda^{\mathbb{N}_0}$ the projection of the points of $W$ onto their values on row zero.\\
\textbf{(i):} If $W$ is generated by $\left(1,0\right)$ then the function  $W_0\rightarrow \Lambda$, in which the values at the indices $\{\left(m,0\right)\}_{m\in {\mathbb{N}_0}}$ of a point in $W$ determine the value in $\left(0,1\right)$, is continuous.\\
\textbf{(ii):} If $W$ is almost-generated then that function is continuous when defined on $\left(W \setminus C_W\right)_0$.\\
\end{proposition}

\begin{proof}
\textbf{(i):} The projection $W \rightarrow W_0$ of the values of a point in $W$ onto its values at $\{\left(m,0\right)\}_{m\in {\mathbb{N}_0}}$ is one-to-one, and as a closed map this is a homeomorphism. Thus the continuous map $\sigma_{\left(0,1\right)}:W\rightarrow W$ induces a continuous map $W_0\rightarrow W_0$. On this map we need only compose the evaluation function at index $0$, which is continuous, to obtain our function.\\

\textbf{(ii):} $C_W$ is a union of fibers of the projection mentioned in (i), thus the restriction of that projection to $W\setminus C_W$ is also a closed map on its image, and thus a homeomorphism. The proof continues as that of (i).\\
\end{proof}

Prop. \ref{generation_continuity} means that for every point (in case (ii), outside of the negligible set $\left(C_W\right)_0$) the value at $\left(1,0\right)$ is already determined by a cylinder of fixed values on row zero from index  $\left(0,0\right)$ up to some index  $\left(n,0\right)$ (with the points in the negligible set withdrawn from it).

 \begin{definition}
    \label{LP_def}
     Given a two-dimensional subshift $W\subseteq \Lambda^{{\mathbb{N}_0}^2}$, if it is generated by $\left(1,0\right)$, we shall say that the generation is carried out in a \textit{left permutative manner} if the function $\Lambda^{\mathbb{N}_0}\rightarrow \Lambda$ in which the values at the indices $\{\left(m,0\right)\}_{m\in {\mathbb{N}_0}}$ determine the value in $\left(0,1\right)$ changes value when the value at $\left(0,0\right)$ of every point in $W$ is changed. If $W$ is almost-generated by $\left(1,0\right)$ we shall say that the generation is done in a \textit{left permutative manner} if the above is true outside of $C_W$.
 \end{definition}

\begin{remark}
\label{LP_rem}
Notice that an alternative definition could be:
"Given a two-dimensional subshift $W\subseteq \Lambda^{{\mathbb{N}_0}^2}$, if it is generated by $\left(1,0\right)$, we shall say that the generation is carried out in a \textit{left permutative manner} if every two such cylinders, as described in the paragraph before the definition, that are identical except in their value at index $\left(0,0\right)$ determine different values at  $\left(0,1\right)$. And similarly for almost-generation with the negligible set $\left(C_W\right)_0$." 
This  definition is a priori stronger than the chosen one, and we shall not need here this additional strength.
\end{remark}

Let us denote by $\xi$ the partition of a two-dimensional subshift $W$ according to the value of the points at $\left(0,0\right)$.\\

\begin{definition}
We shall say that a two dimensional subshift $W\subseteq\Lambda^{{\mathbb{N}_0}^{2}}$ is \textit{reciprocal} if it is almost-generated by both $\left(1,0\right)$ and $\left(0,1\right)$ and satisfies the following property:

There exists some $0<a\in\mathbb{R}$ so that the maximal number of elements of $\vee _{i=0}^{\left\lfloor al\right\rfloor}\sigma_{\left(0,1\right)}^{-i}\mathcal{\xi}$
that any element of $\vee _{i=0}^{l}\sigma_{\left(1,0\right)}^{-i}\mathcal{\xi}$ intersects is asymptotically sub-exponential (i.e. $o\left(e^{\lambda l}\right)$ for all $0<\lambda\in\mathbb{R}$), and also that the maximal number of elements of $\vee _{i=0}^{l}\sigma_{\left(1,0\right)}^{-i}\mathcal{\xi}$
that any element of $\vee _{i=0}^{\left\lfloor al\right\rfloor}\sigma_{\left(0,1\right)}^{-i}\mathcal{\xi}$
intersects is asymptotically sub-exponential.\\

\end{definition}

A priori, it seems a very harsh requirement that exactly the same linear
function $al$ will hold for both directions. However, notice that if it does hold then $a$ must be the ratio of the topological entropies of the projections of $W$ onto the values on each of the two rays ${\mathbb{N}_0}\left(1,0\right)$ and ${\mathbb{N}_0}\left(0,1\right)$. This natural value of $a$ raises hope that reciprocity is not a very rare phenomenon.\\

The first statement of the following lemma generalizes an observation (and proof) made by D. Rudolph in \cite{key-9}.

\begin{lemma}
\label{recip_lemma}
    Let $W\subseteq\Lambda^{{\mathbb{N}_0}^{2}}$ be a two
dimensional reciprocal subshift. Then every Borel probability measure $\mu$ that is invariant relative to both $\sigma_{\left(1,0\right)}$ and $\sigma_{\left(0,1\right)}$
sastisfies $h_{\mu}\left(\sigma_{\left(1,0\right)},\xi\,|\,\mathcal{A}\right)=ah_{\mu}\left(\sigma_{\left(0,1\right)},\xi\,|\,\mathcal{A}\right)$ for every invariant sub-$\sigma$-algebra $\mathcal{A}$ (i.e. $\sigma_{\left(1,0\right)}^{-1}\mathcal{A}\overset{\mod\mu}{=}\sigma_{\left(0,1\right)}^{-1}\mathcal{A}\overset{\mod\mu}{=}\mathcal{A}$), this is true also for the natural invertible extension, and the Pinsker $\sigma$-algebra of $\sigma_{\left(1,0\right)}$ is equal to that of $\sigma_{\left(0,1\right)}$.
\end{lemma}

\begin{proof}

If $\mu$ is composed only of atoms then the claim is trivial. Otherwise, it is only the continuous component of $\mu$ that influences the claim. So we may suppose without loss of generality that $\mu$ is atom-less and hence that the almost-generations are measure-theoretical-generations.

    Consider
\[
H_{\mu}\left(\vee_{i=0}^{l-1}\sigma_{\left(1,0\right)}^{-i}\xi\,|\,\mathcal{A}\right)\leq H_{\mu}\left(\left(\vee_{i=0}^{l-1}\sigma_{\left(1,0\right)}^{-i}\xi\right)\vee\left(\vee_{i=0}^{\left\lfloor al\right\rfloor -1}\sigma_{\left(0,1\right)}^{-i}\xi\right)\,|\,\mathcal{A}\right)
\]
\[
=H_{\mu}\left(\vee_{i=0}^{\left\lfloor al\right\rfloor -1}\sigma_{\left(0,1\right)}^{-i}\xi\,|\,\mathcal{A}\right)+H_{\mu}\left(\vee_{i=0}^{l-1}\sigma_{\left(1,0\right)}^{-i}\xi\,|\,\vee_{i=0}^{\left\lfloor al\right\rfloor -1}\sigma_{\left(0,1\right)}^{-i}\xi\vee\mathcal{A}\right)
\]

\[
\leq H_{\mu}\left(\vee_{i=0}^{\left\lfloor al\right\rfloor -1}\sigma_{\left(0,1\right)}^{-i}\xi\right)+\log\left(asymptotically-sub-exp\right)=H_{\mu}\left(\vee_{i=0}^{\left\lfloor al\right\rfloor -1}\sigma_{\left(0,1\right)}^{-i}\xi\right)+o\left(l\right).
\]
\\

By Dividing both sides by $l$ we get

$\frac{1}{l}H_{\mu}\left(\vee_{i=0}^{l-1}\sigma_{\left(1,0\right)}^{-i}\xi\,|\,\mathcal{A}\right)\leq\frac{al}{l}\frac{1}{al}H_{\mu}\left(\vee_{i=0}^{\left\lfloor al\right\rfloor -1}\sigma_{\left(0,1\right)}^{-i}\xi\,|\,\mathcal{A}\right)+\frac{1}{l}o\left(l\right)$,
and by letting $l\rightarrow\infty$ we conclude $h_{\mu}\left(\sigma_{\left(1,0\right)},\xi\,|\,\mathcal{A}\right)\leq ah_{\mu}\left(\sigma_{\left(0,1\right)},\xi\,|\,\mathcal{A}\right)$ - here we used the fact that the two almost-generations are measure-theoretical-generations.
The reverse inequality is obtained similarly.\\

The proof for the natural invertible extension of the two-dimensional subshift $W$ is identical to the above (in fact, the first claim could have been deduced from the invertible case).\\

As for the equality of the two Pinsker $\sigma$-algebras, the proof is similar to that in the proof of Lemma \ref{pinsker_lemma2}.
\end{proof}

This leads us to the definition of RLP subshifts.

\begin{definition}

We shall say that a two-dimensional subshift $W\subseteq \Lambda^{{\mathbb{N}_0}^2}$ is an \textit{RLP subshift} if it is reciprocal and the almost-generation by $\left(1,0\right)$ is done in a left permutative manner.
    
\end{definition}

We can now formulate the generalized version of Theorem \ref{biper_th}.

\begin{theorem}
\label{RLP_theorem}
    If  $W\subseteq \Lambda^{{\mathbb{N}_0}^2}$ is an RLP subshift and $\mu$ is an atom-less invariant Borel probability measure on it then the function $x \mapsto \mu _x ^{\sigma_{\left(1,0\right)}^{-1}\mathcal{B}}\left(\{x\}\right)$ (where $\mathcal{B}$ is the Borel $\sigma$-algebra of $W$) is $\sigma_{\left(0,1\right)}$-invariant, measurable with respect to the Pinsker $\sigma$-algebras of both maps $\sigma_{\left(1,0\right)}$ and $\sigma_{\left(0,1\right)}$ (which are equal), and when  $\mu _x ^{\sigma_{\left(1,0\right)}^{-1}\mathcal{B}}$ is considered as a probability measure on $\Lambda$ it is $\mu$-almost-surely the uniform probability on a subset $\Lambda_{\mu,x} \subseteq \Lambda$.
\end{theorem}

\begin{proof}
Since $\mu$ is atom-less, almost-generation implies measure-theoretical-generation.\\

The fact that $W$ is measure-theoretically-generated by $\left(1,0\right)$ in a left permutative manner lets us apply Lemma \ref{parry_lemma} (by Remark \ref{parry_remark} following it). Hence $x \mapsto \mu _x ^{\sigma_{\left(1,0\right)}^{-1}\mathcal{B}}\left(\{x\}\right)$ is $\sigma_{\left(0,1\right)}$-invariant. By Lemma \ref{pinsker_lemma1} it is measurable relative to the Pinsker $\sigma$-algebra of $\sigma_{\left(0,1\right)}$ that by Lemma \ref{recip_lemma} is equal to that of $\sigma_{\left(1,0\right)}$.\\

   The last claim follows from the fact that measurability relative the Pinsker $\sigma$-algebra of $\sigma_{\left(1,0\right)}$ implies, in particular, measurability relative to $\sigma_{\left(1,0\right)}^{-1}\mathcal{B}$.
\end{proof}

Also for RLP subshifs we define $\pi_\mu$.

\begin{definition}
Given an RLP subshift $W$ together with an invariant Borel probability measure $\mu$ defined on it, denote $\Lambda_{\mu ,x}$ as in Theorem \ref{RLP_theorem}. We define  $\pi_{\mu} : W \rightarrow \left(2^\Lambda\right)^{{\mathbb{N}_0}}$ as the one-dimensional shift-equivariant map satisfying $\left( \pi_{\mu} \left( x \right) \right)_{0} = \Lambda_{\mu,x}$. We also define  $\pi_{\mu}' : W \rightarrow \left(2^\Lambda\right)^{{\mathbb{N}_0}^2}$ as the two-dimensional shift-equivariant map satisfying $\left( \pi_{\mu} \left( x \right) \right)_{\left(0,0\right)} = \Lambda_{\mu,x}$.\\
\end{definition}

\subsection{The Times 2 Times 3 Cellular Automaton}

Given two coprime positive integers $p$ and $q$ (such as $2$ and $3$) the invertible natural extension of the dynamical system of multiplication of the circle by $p$ and $q$ can be essentially symbolically coded as a full (i.e. also for negative space and time indices) space-time diagram of a cellular automaton in the alphabet $\left\{ 0,1,2,\dots,pq-1\right\} $. The cellular automaton, on the index set $\mathbb{N}$, is that of multiplying a fraction in the unit interval that is expanded in base $pq$ by $p$ - the rule for each index is then dependent on itself and its right neighbor. In its space-time diagram times $p$ is represented by $\sigma_{\left(0,1\right)}$ and times $q$ by $\sigma_{\left(1,-1\right)}$. This system was exploited in Rudolph's proof of his theorem (cf. \cite{key-9} or the survey \cite{key-6}). Let us denote it by $\tilde{X}$, and let us denote by $X$ its factor composed of restricting every point to the indices of the quadrant generated by $\left(0,1\right)$ and $\left(1,-1\right)$ and then changing coordinates linearly by transfering $\left(1,-1\right)\mapsto \left(1,0\right)$ and $\left(0,1\right)\mapsto \left(0,1\right)$ (so the full one-dimensional shift, representing the fraction in base $pq$ now corresponds to the values on the diagonal).\\

$X$ factors onto the circle with its multiplication maps $S_{p}$ and $S_{q}$ - multiplication by $p$ and $q$ respectively - by expanding in base $pq$ the digits that appear on the diagonal, that is $x\mapsto \sum_{i=0}^\infty x_{\left(i,i\right)}\left(pq\right)^{-i-1}$. By this relation to the circle it is seen that both $\sigma_{\left(1,0\right)}$ and $\sigma_{\left(0,1\right)}$ almost-generates $X$, where a countable set outside which there is generation can be taken to be the inverse image of the $pq$-adic fractions in the circle. Moreover, since $p$ and $q$ are coprime this is done in a left-permutative manner: $S_{p}|_{S_{q}^{-1}\left(\left\{ x\right\} \right)}:S_{q}^{-1}\left(\left\{ x\right\} \right)\rightarrow S_{q}^{-1}\left(\left\{ S_{p}\left(x\right)\right\} \right)$ is bijective and hence left-permutativity  holds (also the a priori stronger alternative definition appearing in Remark  \ref{LP_rem}  holds since any two non-empty cylinders determining the values at indices $\left(0,0\right),\left(1,0\right),\dots,\left(n,0\right)$ (for some $n>0$) and possessing identical values for the indices $\left(1,0\right),\dots,\left(n,0\right)$ have a non-empty intersection).\\

As Rudolph observed, $X$ is reciprocal (he did not use that term). This again is seen by considering the circle. Denote by $\alpha$ the partition $\left\{ \left[\frac{j}{pq},\frac{j+1}{pq}\right)\right\} _{j=0}^{pq-1}$. The lengths of the intervals composing the partition $\vee_{i=0}^{l-1}S_{p}^{-i}\left(\alpha\right)$
is $\frac{1}{p^{l}q}$, and of $\vee_{i=0}^{m-1}S_{q}^{-i}\left(\alpha\right)$ is $\frac{1}{pq^{m}}$. Choosing $m\left(l\right)=\left\lfloor \log_{q}p^{l-1}\right\rfloor =\left\lfloor \frac{\left(l-1\right)\log p}{\log q}\right\rfloor $ the length of the intervals composing $\vee_{i=0}^{m\left(l\right)-1}S_{q}^{-i}\left(\alpha\right)$
is between $\frac{1}{p^{l}q}$ and $\frac{1}{p^{l}}$, and hence none of which can intersect more than $q+1$ of the intervals composing $\vee_{i=0}^{l-1}S_{p}^{ i}\left(\alpha\right)$, and none of the latters can intersect more than $2$ of the formers.\\

Apart from this example, its generalization (to higher dimensional tori etc.) and bi-permutative cellular automata, there are no other proven examples of RLP subshifts in our disposal. Examining whether subshifts satisfy the reciprocity requirement seems a challenging task.

\subsection{Invariant Mearures of RLP Subshifts with a Structural Factor of Zero Entropy}

In this section we state the main generalized claims. The proofs are mostly similar to the original ones. When they are not we point out the differences.\\

A few definitions first.\\

We continue to denote the component-wise projection $\left(A,a\right)\mapsto A$ by $\psi_{1}: \{ \left(A,a\right)\,:\,A\subseteq\Lambda,\,a\in A\} ^{{\mathbb{N}_0}}\rightarrow\left(2^\Lambda\right)^{{\mathbb{N}_0}}$ and  $\left(A,a\right)\mapsto a$ by  $\psi_{2}:\left\{ \left(A,a\right)\,:\,A\subseteq\Lambda,\,a\in A\right\} ^{{\mathbb{N}_0}}\rightarrow \Lambda^{{\mathbb{N}_0}}$.\\

We also introduce new similar maps.\\

\begin{definition}
We denote by 

$\psi_{1}':\left\{ \left(A,a\right)\,:\,A\subseteq\Lambda,\,a\in A\right\} ^{{\mathbb{N}_0}^2}\rightarrow\left(2^\Lambda\right)^{{\mathbb{N}_0}^2}$
the projection $\left(A,a\right)\mapsto A$,

and by $\psi_{2}':\left\{ \left(A,a\right)\,:\,A\subseteq\Lambda,\,a\in A\right\} ^{{\mathbb{N}_0}^2}\rightarrow \Lambda^{{\mathbb{N}_0}^2}$
the projection $\left(A,a\right)\mapsto a$.   \\   
\end{definition}

\begin{definition}
Given an RLP subshift $W$ with an atom-less invariant measure $\mu$, we denote by $\varphi_\mu$ the
map from $W$ to $\left\{ \left(A,a\right)\,:\,A\subseteq\Lambda,\,a\in A\right\} ^{{\mathbb{N}_0}}$
defined $\mu$-almost-surely as $\left(\varphi_\mu\left(x\right)\right)_{i}=\left(\left(\pi_{\mu}\left(x\right)\right)_{i,0},x_{i,0}\right)$.\\    
\end{definition}

\begin{definition}
    Given a two-dimensional subshift $W\subseteq \Lambda^{{\mathbb{N}_0}^2}$ that is almost-generated by $\left(1,0\right)$, consider the Borel subset of $\left(2^\Lambda \setminus \{\emptyset\} \right)^{{\mathbb{N}_0}^2}$ which is the set of all points $z\in \left(2^\Lambda  \setminus \{\emptyset\} \right)^{{\mathbb{N}_0}^2}$ satisfying that for every $a\in z_{i,j+1}$ there exists $x \in W\setminus \sigma_{\left(i,j\right)}^{-1} \left(C_W\right)$ so that $x_{i,j+1} = a$ and $x_{k,j}\in z_{k,j}$ for every $k\geq i$. We define $Z_W$ to be the collection of all points  $z$ of that Borel subset that \underline{also} satisfy the following two requirements:
    \begin{enumerate}
    \item that for any given $i\in {\mathbb{N}_0}$ all elements of the collection $\{ z_{i,j}\}_{i\in {\mathbb{N}_0}}$ are subsets of $\Lambda$ of the same size.
    \item $\psi_2'\left(\left(\psi_1'\right)^{-1}\left(\{z\}\right)\right) \subseteq W$.\\
    \end{enumerate}
\end{definition}

\begin{remark}
   \label{not_compact_gen_rem}
    Albeit not being compact notice that $\sigma_{\left(1,0\right)}\left(Z_W\right),\sigma_{\left(0,1\right)}\left(Z_W\right) \subseteq Z_W$, and hence also $\left(\psi_1'\right)^{-1}\left(Z_W\right)$ shares this invariance. Moreover,  the values of a point of $Z_W$ on row zero determine all the values above them, and for any $z\in Z_W$ the points of the set $\left(\psi_1'\right)^{-1}\left(\{z\}\right)$ also satisfy this property apart from a countable subset.\\
\end{remark}

\begin{theorem}
\label{RLP_zero_ent_analysis}
Given an RLP subshift $W$ with an atom-less invariant measure $\mu$, assume also $h_{\left( \pi_{\mu} \right)_{*}\mu}\left(\sigma\right)=0$. We define a Borel probability measure $\tilde{\mu}$ on $\{ \left(A,a\right)\,:\,A\subseteq\Lambda,\,a\in A\} ^{{\mathbb{N}_0}}$ by requiring that for a cylinder depending on the first values
\[\tilde{\mu}\left(\left[  \left(A_{0},r_{0}\right), \left(A_{1},r_{1}\right),\dots ,\left(A_{n},r_{n}\right) \right]\right) = \frac{\left( \pi_{\mu} \right)_{*}\mu\left(\left[ A_{0}, A_{1},\dots ,A_{n} \right]\right)}{\left|A_{0}\right| \left|A_{1}\right| \cdots \left|A_{n}\right|}.\]
Then:\\ 
\textbf{(i):} $\varphi_\mu$ is a measure-theoretical isomorphism between the dynamical systems of $\left( W,\sigma_{\left(1,0\right)},\mu\right)$ and $\left(\left\{ \left(A,a\right)\,:\,A\subseteq\Lambda,\,a\in A\right\} ^{{\mathbb{N}_0}},\sigma,\tilde{\mu}\right)$, i.e. $\left(\varphi_\mu\right)_*\mu = \tilde{\mu}$.\\
\textbf{(ii):} The one-dimensional shift space $\left(2^{\Lambda}\right)^{{\mathbb{N}_0}}$ together with $\left( \pi_{\mu} \right)_{*}\mu$ is the Pinsker factor of $\left(W,\sigma_{\left(1,0\right)},\mu\right)$.\\
\textbf{(iii):} With respect to $\left( \pi'_{\mu} \right)_{*}\mu$, the two-dimensional shift space $\left(2^{\Lambda}\right)^{{\mathbb{N}_0}^2}$ is measure-theoretically-generated by $\left(1,0\right)$.\\
\textbf{(iv):} $\left(\pi'_\mu\right)_*\mu \left( Z_W \right) = 1$ .
\end{theorem}

 \begin{proof}
     The proof is similar to that of Theorem \ref{zero_ent_analysis}.  
 \end{proof}

\begin{remark}
    In relation to (iv), we do not know whether the statement would remain true if $Z_W$ were to be defined  with the requirement that in each of its points all components were sets of equal size (not only the ones along the same column).
\end{remark}

This was the analysis theorem of $\mu$. For the proof of the synthesis theorem we shall need the following proposition.\\

\begin{proposition}
\label{ZW_continuity}
For a two-dimensional subshift $W\subseteq \Lambda^{{\mathbb{N}_0}^2}$ that is almost-generated by $\left(1,0\right)$, denote by $\left(Z_W\right)_0 \subseteq \left( 2^\Lambda\right)^{\mathbb{N}_0} $ the projection of the points of $Z_W$ onto their values on row zero. Then the function $\left(Z_W\right)_0 \rightarrow 2^\Lambda$ in which the values at the indices $\{\left(m,0\right)\}_{m\in {\mathbb{N}_0}}$ of a point in $Z_W$ determine the value in $\left(0,1\right)$, is continuous.
\end{proposition}

\begin{proof}
Let $z$ be a point in $Z_W$. By Prop. \ref{generation_continuity}, for every $a\in z_{0,1}$ there exists a finite word $\left(a_0,a_1,\dots,a_n\right)\in z_{0,0}\times z_{1,0}\times \dots \times z_{n,0}$ that when placed at the indices $\left(0,0\right),\left(1,0\right),\dots,\left(n,0\right)$ determines in $W$ the value $a$ at index $\left(0,1\right)$. For $n'$ which is the maximal such $n$ when varying $a$ in $z_{0,1}$, we then have that in $Z_W$ the word $\left(z_{0,0},z_{1,0},\dots,z_{n',0}\right)$ at indices $\left(0,0\right),\left(1,0\right),\dots,\left(n',0\right)$ determines the value $z_{0,1}$ at the index $\left(0,1\right)$ (because of the constant size along columns requirement in the definition of $Z_W$).
\end{proof}

\begin{theorem}
\label{RLP_zero_ent_synth}

Let $W\subseteq \Lambda^{{\mathbb{N}_0}^2}$ be a two-dimensional subshift that is almost-generated by $\left(1,0\right)$ in a left permutative manner, and let $\nu$ be a Borel probability measure on $Z_W$ which is invariant relative to the two-dimensional shift. Form a measure $\tilde{\mu}$ on $\left(\psi_1'\right)^{-1}\left(Z_W\right) \subseteq \left\{ \left(A,a\right)\,:\,A\subseteq\Lambda,\,a\in A\right\} ^{{\mathbb{N}_0}^2}$ by requiring that for a cylinder depending on the first values on row zero 
\[\tilde{\mu}\left(\left[  \left(A_{0,0},a_{0,0}\right), \left(A_{1,0},a_{1,0}\right),\dots ,\left(A_{n,0},a_{n,0}\right) \right] \cap \left(\psi_1'\right)^{-1}\left(Z_W\right) \right)\]
\[
= \frac{\nu\left(\left[ A_{0,0}, A_{1,0},\dots ,A_{n,0} \right]\right)}{\left|A_{0,0}\right| \left|A_{1,0}\right| \cdots \left|A_{n,0}\right|},
\]
(cf. Remark  \ref{not_compact_gen_rem}) and define $\mu = \left(\psi_2'\right)_*\tilde{\mu}$. Then:\\
\textbf{(i):} $\mu$ is invariant relative to the two-dimensional shift.\\
\textbf{(ii):} If $W$ is an RLP-subshift and $h_\nu\left( \sigma_{\left(1,0\right)} \right) = 0$ then $\left(\pi'_{\mu}\right)_*\mu = \nu$.\\
\textbf{(iii):} If in addition to the premises of (ii),  $\nu$ is also ergodic relative to the two-dimensional shift then so is $\mu$.\\
\end{theorem}

\begin{proof}
Only the generalized proof of (i) deserves to be written down.\\

\textbf{(i):} It suffices to prove the invariance of $\tilde{\mu}$ relative to the two-dimensional shift. As its invariance relative to $\sigma_{\left(1,0\right)}$ is immediate, it is left to prove its invariance relative to $\sigma_{\left(0,1\right)}$.\\

We fix some cylinder $\left[  \left(A_{0,0},a_{0,0}\right), \left(A_{1,0},a_{1,0}\right),\dots ,\left(A_{d,0},a_{d,0}\right) \right]$ of the first $d+1$ values on row zero, and aim to prove that its inverse image through $\sigma_{\left(0,1\right)}|_{\left(\psi_1'\right)^{-1}\left(Z_W\right) }$ is of equal $\tilde{\mu}$-measure.\\

The $\tilde{\mu}$-probability of  $\left[  \left(A_{0,0},a_{0,0}\right), \left(A_{1,0},a_{1,0}\right),\dots ,\left(A_{d,0},a_{d,0}\right) \right]$ is equal to 
\[  \frac{\nu\left(\left[A_{0,0},\dots,A_{d,0}\right]\right)}{\left|A_{0,0}\right|\cdots\left|A_{d,0}\right|}. \] \\ 

By the invariance of $\nu$  relative to $\sigma_{\left(0,1\right)}$ this is equal to 
\[
\frac{\nu\left(\sigma_{\left(0,1\right)}|_{Z_W}^{-1}\left(\left[A_{0,0},\dots,A_{d,0}\right]\right)\right)}{\left|A_{0,0}\right|\cdots\left|A_{d,0}\right|}.
\]

However, by Prop. \ref{ZW_continuity},  denoting 
\[
J_N = \{\left(B_{0,0},\dots,B_{N,0}\right)  :\,\forall  0\leq i \leq d  \,\, \left[B_{i,0},\dots,B_{N,0}\right] \cap Z_W  \,\subseteq \, \sigma_{\left(0,1\right)}|_{Z_W}^{-1}\left(\left[A_{i,0},\dots,A_{d,0}\right]\right) \},
\]
this expression is equal to the following expression:\\

\[
\lim_{N\rightarrow \infty}\sum_{\left(B_{0,0},\dots,B_{N,0}\right) \in J_N } \frac{\nu \left( \left[B_{0,0},\dots,B_{N,0}\right] \right)}{\left|A_{0,0}\right|\cdots\left|A_{d,0}\right|} \] \\

\[
= \lim_{N\rightarrow \infty} \sum_{ \left(B_{0,0},\dots,B_{N,0}\right) \in J_N} \frac{\nu \left( \left[B_{0,0},\dots,B_{N,0}\right] \right)}{\left|B_{0,0}\right|\cdots\left|B_{d,0}\right|}  \] \\

\[ 
=  \lim_{N\rightarrow \infty} \sum_{ \left(B_{0,0},\dots,B_{n,0}\right) \in J_N} \left|B_{d+1,0}\right|\cdots \left|B_{N,0}\right| \frac{ \nu \left( \left[B_{0,0},\dots,B_{N,0} \right] \right)}{\left|B_{0,0}\right|\cdots\left|B_{N,0} \right|}.
\]\\

On the other hand, denoting

\[
I_N = \{\left(\left(B_{0,0},b_{0,0}\right),\dots,\left(B_{N,0},b_{N,0}\right)\right)\  :\, \left[\left(B_{0,0},b_{0,0}\right),\dots,\left(B_{N,0},b_{N,0}\right)\right] \cap {\left(\psi_1'\right)^{-1}\left(Z_W\right) } \cap \left(  \psi_2' \right)^{-1} \left(  W\setminus C_W \right)   \]

\[
 \subseteq \sigma_{\left(0,1\right)}|_{\left(\psi_1'\right)^{-1}\left(Z_W\right) }^{-1}\left(\left[  \left(A_{0,0},a_{0,0}\right), \left(A_{1,0},a_{1,0}\right),\dots ,\left(A_{d,0},a_{d,0}\right) \right]\right)\},
\]

Props. \ref{generation_continuity} and \ref{ZW_continuity} imply that the $\tilde{\mu}$-probability of 
\[
\sigma_{\left(0,1\right)}|_{\left(\psi_1'\right)^{-1}\left(Z_W\right) }^{-1}\left(    \left[  \left(A_{0,0},a_{0,0}\right), \left(A_{1,0},a_{1,0}\right),\dots ,\left(A_{d,0},a_{d,0}\right) \right]  \right)
\]
 is equal to

\[ \lim_{N\rightarrow \infty}
\sum_{\left(\left(B_{0,0},b_{0,0}\right),\dots,\left(B_{N,0},b_{N,0}\right)\right) \in I_N} \frac{ \nu \left( \left[B_{0,0},\dots,B_{N,0} \right] \right)}{\left|B_{0,0}\right|\cdots\left|B_{N,0} \right| }.
\]\\\\

Let $\varepsilon>0$.\\

We may assume without loss of generality that $\nu$ is ergodic relative to the two-dimensional shift. By that, we may further suppose without loss of generality that $\nu$-almost surely row zero has infinite many components which are not singletons (in fact, a positive density of such), for otherwise the proof is trivial. Given such a $\left(B_{0,0},B_{1,0},\dots\right) \in Z_W$ , we shall refer to the probability on the  atom $\left[B_{0,0},B_{1,0},\dots\right]\in\left(\psi_1'\right)^{-1}{\mathcal{B}_{Z_W}}$ that is the probability of independent uniformly random choices on $B_{0,0},B_{1,0},\dots$ (the values of a point on row zero almost-surely determine the rest of its values). Then by Prop. \ref{generation_continuity}, we know that a point $\left(b_{0,0},b_{1,0},\dots\right) \in \left[B_{0,0},B_{1,0},\dots\right]$  will almost-surely satisfy for a large enough $N$ and for every $0\leq i\leq d$ that$\left[*,\dots,*,b_{i,0},b_{i+1,0},\dots b_{N,0}\right]$ determines the values at indices $\left(i,1\right),\left(i+1,1\right),\dots \left(d,1\right)$ for points inside $\left( W\setminus \sigma_{\left(1,0\right)}^{-i}\left(C_W\right)\right)$, and similarly if we change the values at the first $d+1$ places in $\left[b_{0,0},b_{1,0},\dots b_{N,0}\right]$ to any other option in $B_{0,0}\times B_{1,0}\times\dots \times B_{d,0}$  (notice that by left permutativity exactly one of those $\left|B_{0,0}\right|\times \left|B_{1,0}\right|\times\dots \times \left|B_{d,0}\right|$ length $N+1$ cylinders  will determine the values $\left(a_{0,0},\dots,a_{d,0}\right)$ above it). As $N\rightarrow \infty$ the  probability that $N$ will indeed be large enough approaches unity.\\

 Therefore, for a large enough $N$ more than $\left(1-\varepsilon\right) \left|B_{d+1,0}\right|\times \left|B_{d+2,0}\right|\times\dots \times \left|B_{N,0}\right|$ of the elements $\left(b_{d+1,0},b_{d+2,0},\dots b_{N,0}\right)\in B_{d+1,0}\times B_{d+2,0}\times\dots \times B_{N,0}$ satisfy that every $\left(b'_{0,0},b_{1,0},\dots,b'_{d,0},b_{d+1,0},b_{d+2,0},\dots b_{N,0}\right)\in B_{0,0}\times \times\dots \times B_{N,0}$ forms a cylinder that determines the values at indices $\left(0,1\right),\left(1,1\right),\dots \left(d,1\right)$ inside $W\setminus C_W$, and exactly one of these $\left|B_{0,0}\right|\times \left|B_{1,0}\right|\times\dots \times \left|B_{d,0}\right|$  length $N+1$ cylinders determines the values  $\left(a_{0,0},\dots,a_{d,0}\right)$ above it.\\\\

We have just performed an analysis for a fixed $\left(B_{0,0},B_{1,0},\dots\right) \in Z_W$. However, as $N\rightarrow \infty$ the $\nu$-probability that $N$ will indeed be large enough for the preceding paragraph to hold for an element of $Z_W$  approaches unity.\\\\

Given $\left(B_{0,0},\dots,B_{N,0}\right)$  that satisfies $\left[B_{0,0},\dots,B_{N,0}\right] \cap Z_W  \,\subseteq \, \sigma_{\left(0,1\right)}|_{Z_W}^{-1}\left(\left[A_{0,0},\dots,A_{d,0}\right]\right)$, we denote
\[
r_{\left(B_{0,0},\dots,B_{N,0}\right),N} := \left|  \{ \left(b_{0,0},\dots,b_{N,0}\right)  : \, \left(\left(B_{0,0},b_{0,0}\right),\dots,\left(B_{N,0},b_{N,0}\right)\right) \in I_N \} \right|.
\]

Thus 

\[
\lim_{N\rightarrow \infty}
\sum_{\left(\left(B_{0,0},b_{0,0}\right),\dots,\left(B_{N,0},b_{N,0}\right)\right) \in I_N} \frac{ \nu \left( \left[B_{0,0},\dots,B_{N,0} \right] \right)}{\left|B_{0,0}\right|\cdots\left|B_{N,0} \right| }
\]
\[
= 
\lim_{N\rightarrow\infty}
\sum_{  \{\left(B_{0,0},\dots,B_{N,0}\right)\  :\,  \left[B_{0,0},\dots,B_{N,0}\right] \cap Z_W  \,\subseteq \, \sigma_{\left(0,1\right)}|_{Z_W}^{-1}\left(\left[A_{0,0},\dots,A_{d,0}\right]\right) \}} r_{\left(B_{0,0},\dots,B_{N,0}\right),N}  \frac{ \nu \left( \left[B_{0,0},\dots,B_{N,0} \right] \right)}{\left|B_{0,0}\right|\cdots\left|B_{N,0} \right|}
\]
\[
\leq 
\lim_{N\rightarrow\infty} \left(\right.
\sum_{  \{\left(B_{0,0},\dots,B_{N,0}\right)\  :\,  \left[B_{0,0},\dots,B_{N,0}\right] \cap Z_W  \,\subseteq \, \sigma_{\left(0,1\right)}|_{Z_W}^{-1}\left(\left[A_{0,0},\dots,A_{d,0}\right]\right) \}}  \left|B_{d+1,0}\right|\cdots \left|B_{N,0}\right|   \frac{   \nu \left( \left[B_{0,0},\dots,B_{N,0} \right] \right)}{\left|B_{0,0}\right|\cdots\left|B_{N,0} \right|}
\]

\[
+
\sum_{  \{\left(B_{0,0},\dots,B_{N,0}\right)\  :\,  \left[B_{0,0},\dots,B_{N,0}\right] \cap Z_W  \,\subseteq \, \sigma_{\left(0,1\right)}|_{Z_W}^{-1}\left(\left[A_{0,0},\dots,A_{d,0}\right]\right) \}} \varepsilon \left|B_{0,0}\right|\cdots \left|B_{d,0}\right|  \frac{   \nu \left( \left[B_{0,0},\dots,B_{N,0} \right] \right)}{\left|B_{0,0}\right|\cdots\left|B_{N,0} \right|}\left.\right)
\]

\[
\leq 
\lim_{N\rightarrow\infty} \left(\right.
\sum_{  \{\left(B_{0,0},\dots,B_{N,0}\right)\  :\,  \left[B_{0,0},\dots,B_{N,0}\right] \cap Z_W  \,\subseteq \, \sigma_{\left(0,1\right)}|_{Z_W}^{-1}\left(\left[A_{0,0},\dots,A_{d,0}\right]\right) \}}  \left|B_{d+1,0}\right|\cdots \left|B_{N,0}\right|   \frac{   \nu \left( \left[B_{0,0},\dots,B_{N,0} \right] \right)}{\left|B_{0,0}\right|\cdots\left|B_{N,0} \right|}
\]

\[
+
\varepsilon \left|\Lambda \right|^{d+1}  \sum_{  \{\left(B_{0,0},\dots,B_{N,0}\right)\  :\,  \left[B_{0,0},\dots,B_{N,0}\right] \cap Z_W  \,\subseteq \, \sigma_{\left(0,1\right)}|_{Z_W}^{-1}\left(\left[A_{0,0},\dots,A_{d,0}\right]\right) \}}   \frac{   \nu \left( \left[B_{0,0},\dots,B_{N,0} \right] \right)}{\left|B_{0,0}\right|\cdots\left|B_{N,0} \right|}\left.\right).
\]

Since $\varepsilon>0$ was arbitrary we obtain

\[
\lim_{N\rightarrow \infty}
\sum_{\left(\left(B_{0,0},b_{0,0}\right),\dots,\left(B_{N,0},b_{N,0}\right)\right) \in I_N} \frac{ \nu \left( \left[B_{0,0},\dots,B_{N,0} \right] \right)}{\left|B_{0,0}\right|\cdots\left|B_{N,0} \right| }
\]

\[
\leq
\lim_{N\rightarrow\infty} 
\sum_{  \{\left(B_{0,0},\dots,B_{N,0}\right)\  :\,  \left[B_{0,0},\dots,B_{N,0}\right] \cap Z_W  \,\subseteq \, \sigma_{\left(0,1\right)}|_{Z_W}^{-1}\left(\left[A_{0,0},\dots,A_{d,0}\right]\right) \}}  \left|B_{d+1,0}\right|\cdots \left|B_{N,0}\right|   \frac{   \nu \left( \left[B_{0,0},\dots,B_{N,0} \right] \right)}{\left|B_{0,0}\right|\cdots\left|B_{N,0} \right|},
\]
and one can show similarly that

\[
\lim_{N\rightarrow \infty}
\sum_{\left(\left(B_{0,0},b_{0,0}\right),\dots,\left(B_{N,0},b_{N,0}\right)\right) \in I_N} \frac{ \nu \left( \left[B_{0,0},\dots,B_{N,0} \right] \right)}{\left|B_{0,0}\right|\cdots\left|B_{N,0} \right| }
\]

\[
\geq
\lim_{N\rightarrow\infty} 
\sum_{  \{\left(B_{0,0},\dots,B_{N,0}\right)\  :\,  \left[B_{0,0},\dots,B_{N,0}\right] \cap Z_W  \,\subseteq \, \sigma_{\left(0,1\right)}|_{Z_W}^{-1}\left(\left[A_{0,0},\dots,A_{d,0}\right]\right) \}}  \left|B_{d+1,0}\right|\cdots \left|B_{N,0}\right|   \frac{   \nu \left( \left[B_{0,0},\dots,B_{N,0} \right] \right)}{\left|B_{0,0}\right|\cdots\left|B_{N,0} \right|}.
\]
So we deduce that

\[
\lim_{N\rightarrow \infty}
\sum_{\left(\left(B_{0,0},b_{0,0}\right),\dots,\left(B_{N,0},b_{N,0}\right)\right) \in I_N} \frac{ \nu \left( \left[B_{0,0},\dots,B_{N,0} \right] \right)}{\left|B_{0,0}\right|\cdots\left|B_{N,0} \right| }
\]

\[
=
\lim_{N\rightarrow\infty} 
\sum_{ \left(B_{0,0},\dots,B_{N,0}\right) \in J_N}  \left|B_{d+1,0}\right|\cdots \left|B_{N,0}\right|   \frac{   \nu \left( \left[B_{0,0},\dots,B_{N,0} \right] \right)}{\left|B_{0,0}\right|\cdots\left|B_{N,0} \right|}.
\]\\

As we have seen in the early stage of the proof, the last expression is equal to

\[
\frac{\nu\left(\sigma_{\left(0,1\right)}|_{Z_W}^{-1}\left(\left[A_{0,0},\dots,A_{d,0}\right]\right)\right)}{\left|A_{0,0}\right|\cdots\left|A_{d,0}\right|}.
\]
\end{proof}

Again, we sum up the two theorems in a corollary.\\

\begin{corollary}
\label{RLP_1-1_cor}
Let $W$ be an  RLP subshift . Then there exists a bijective correspondence between the collection of Borel probability measures $\mu$ on $W$ that are invariant and ergodic relative to the two-dimensional shift and satisfy $h_{\left(\pi_{\mu}\right)_*\mu} \left(\sigma\right) = 0$, and the collection of Borel probability measures $\nu$ on $Z_W$ that are invariant and ergodic relative to the two-dimensional shift and satisfy $h_{\nu} \left(\sigma_{\left(1,0\right)}\right) = 0$.

\end{corollary}

Einstein Institute of Mathematics, Edmond J. Safra campus, The Hebrew
University of Jerusalem, Israel.\\

matan.tal@mail.huji.ac.il


\begin{thebibliography}{1}

\bibitem{boyle}M. Boyle. Open problems in symbolic dynamics. Contemporary mathematics, 469, 69-118 (2008).

\bibitem{key-7}M. Einsiedler, T. Ward. Entropy Geometry and Disjointness for Zero-Dimensional Algebraic Actions. Journal für die Reine und Angewandte Mathematik, 584, 195-214  (2005).

\bibitem{key-4}M. Einsiedler. Invariant Subsets and Invariant Measures
for Irreducible Actions on Zero-Dimensional Groups. Bulletin of the
London Mathematical Society, 36(3), 321-331 (2004).

\bibitem{key-5}M. Einsiedler. Isomorphism and Measure Rigidity for
Algebraic Actions on Zero-Dimensional Groups. Monatshefte für Mathematik,
144(1), 39-69 (2005).

\bibitem{eisner}T. Eisner, B. Farkas, M. Haase, R. Nagel. Operator Theoretic Aspects of Ergodic Theory. Graduate Texts in Mathematics 272, Springer (2015).

\bibitem{key-8}B. Host, A. Maass, S. Martínez. Uniform Bernoulli Measure in Dynamics of Permutative Cellular Automata with Algebraic Local Rules. Discrete and Continuous Dynamical Systems, 9(6), 1423-1446
(2003).

\bibitem{kari-taati}J. Kari, S. Taati. Statistical mechanics of surjective cellular automata. Journal of Statistical Physics, 160(5), 1198-1243 (2015).

\bibitem{key-10} T. Meyerovitch, R. Pavlov. On Independence and Entropy for High-Dimensional Isotropic Subshifts. Proceedings of the London Mathematical Society, 109(4), 921-945 (2014).

\bibitem{key-2}W. Parry. \textit{Squaring and Cubing the Circle -
Rudolph's Theorem}. Ergodic Theory of $\mathbb{Z}^{d}$ Actions, Cambridge
Univ. Press, 177--183 (1996).

\bibitem{key-1}M. Pivato. Invariant Measures for Bipermutative Cellular Automata, Discrete and Continuous Dynamical Systems, Volume 12, Number 4, 723-736 (2005).

\bibitem{piv-ref3}M. Pivato. Module shifts and measure rigidity in linear cellular automata. Ergodic Theory and Dynamical Systems, 28(6), 1945-1958 (2008).

\bibitem{piv-ref2}M. Pivato. The ergodic theory of cellular automata. International Journal of General Systems, 41(6), 583-594 (2012).

\bibitem{key-9} D.J. Rudolph. × 2 and × 3 Invariant Measures and Entropy. Ergodic Theory and Dynamical Systems, 10(2), 395-406 (1990).

\bibitem{sablik} M. Sablik. Measure rigidity for algebraic bipermutative cellular automata. Ergodic Theory and Dynamical Systems, 27(6), 1965-1990 (2007).

\bibitem{sobottka}M. Sobottka. Right-permutative cellular automata on topological Markov chains. Discrete and Continuous Dynamical Systems, 2008, 20(4): 1095-1109 (2008).

\bibitem{key-6}M. Tal. Furstenberg's Times 2, Times 3 Conjecture
(a Short Survey). arXiv preprint arXiv:2110.05989 (2021).\\

\end{thebibliography}
\end{document}